\documentclass[final]{article}

\usepackage{graphicx}
\graphicspath{{figures/}{./}}
\usepackage{amsmath,amsfonts,amssymb}
\usepackage{pgf,tikz}

\usepackage{tikz}
\usetikzlibrary{arrows}
\usetikzlibrary{positioning}
\usetikzlibrary{shapes.geometric}
\usetikzlibrary{calc}
\usetikzlibrary{decorations.pathmorphing}
\tikzset{snake it/.style={decorate, decoration=snake, segment length=1cm}}

 \usepackage{amsmath,amsfonts,amssymb,amsthm}
 \newcommand{\ETDS}[4]{}
 \newcommand{\runningheads}[2]{}
 \newcommand{\address}[1]{}
 \newcommand{\recd}[1]{}
 \newcommand{\toprule}{\hline}

\usepackage[utf8]{inputenc}
\usepackage[english]{babel}
\usepackage{fixme}
\usepackage[notref,notcite]{showkeys}
\usepackage{geometry}

\newtheorem{thm}{Theorem}[section]
\newtheorem{prop}[thm]{Proposition}
\newtheorem{lem}[thm]{Lemma}
\newtheorem{cor}[thm]{Corollary}

\newtheorem{definition}[thm]{Definition}
\newtheorem{example}[thm]{Example}

\newtheorem{remark}[thm]{Remark}

\DeclareMathOperator{\Id}{Id}
\DeclareMathOperator{\Perm}{Perm}
\DeclareMathOperator{\Ad}{Ad}

\newcommand{\ordv}[1]{{\omega^0_{#1}}}
\newcommand{\orde}[1]{{\omega^1_{#1}}}
\newcommand{\MESv}[1]{{M_{#1}}}
\newcommand{\MESe}[1]{{M^1_{#1}}}
\newcommand{\enumref}[1]{(\ref{#1})}

\newcommand{\setof}[2]{\left\{#1\,\middle|\,#2\right\}}

\newcommand{\arc}{\!\rightarrow\!}
\newcommand{\larc}[1]{\!\xrightarrow{#1}\!}
\newcommand{\llarc}[2]{\!\xrightarrow{#1,#2}\!}
\newcommand{\supth}{^{\text{th}}}

\newcommand{\N}{\mathbb{N}}
\newcommand{\Z}{\mathbb{Z}}

\newcommand{\C}{\mathbb{C}}

\renewcommand{\AA}{\mathcal{A}}

\newcommand{\DD}{\mathcal{D}}

\newcommand{\LL}{\mathcal{L}}
\newcommand{\OO}{\mathcal{O}}

\newcommand{\GG}{\mathcal{G}}

\newcommand{\X}{\mathsf{X}}

\newcommand{\ZZ}{\mathsf{Z}}

\newcommand{\completeEtau}{\Phi_1}
\newcommand{\completeletter}{\Phi_0}

\newcommand{\completeOtau}[1]{\Psi_1\!\left(#1\right)}

\newcommand{\Nsources}[1]{\alpha_{N}}

\newcommand{\addedge}[2]{#1\oplus #2}
\newcommand{\disjointunion}{\biguplus}

\newcommand{\f}{\!f\!}

\newcommand{\map}[1]{\mathsf{#1}}
\newcommand{\source}{\map{s}}
\newcommand{\range}{\map{r}}
\newcommand{\predicate}[1]{\map{#1}}
\newcommand{\isvalid}[1]{\predicate{valid}\left({#1}\right)}
\newcommand{\isvalidb}[1]{\predicate{valid'}\left(#1\right)}

\title{Visualizing Automorphisms of Graph Algebras} 
\author{
James Emil Avery, Rune Johansen, and Wojciech Szyma{\'n}ski}
\address{
James Emil  Avery \\
Niels Bohr Institute \\
University of Copenhagen Blegdamsvej 17 \\
2100 K\o benhavn \O \\
Denmark \\
\email{avery@nbi.ku.dk} \\[0.5em]

Rune Johansen \\
Department of Mathematics\\
University of Copenhagen\\
Universitetsparken 5\\
2100 K\o benhavn \\
Denmark \\
\email{rune@math.ku.dk} \\[0.5em]

Wojciech Szyma{\'n}ski \\
Department of Mathematics and Computer Science \\
University of Southern Denmark \\
Campusvej 55 \\ 
5230 Odense M \\
Denmark \\
\email{szymanski@imada.sdu.dk}
}
\date{\today}

\begin{document}
%

\maketitle
\begin{abstract}
  Graph $C^\ast$-algebras have been celebrated as $C^\ast$-algebras
  that can be seen, because many important properties may be
  determined by looking at the underlying graph.
  This paper introduces the permutation graph for a permutative endomorphism of a graph $C^\ast$-algebra as a labeled directed multigraph that gives a visual representation of the endomorphism and facilitates computations.
  Combinatorial criteria have previously been developed for deciding when such
  an endomorphism is an automorphism, but here the question is
  reformulated in terms of the permutation graph and new proofs are
  given. Furthermore, it is shown how to use permutation graphs to
  efficiently generate exhaustive collections of permutative
  automorphisms. Permutation graphs provide a natural link to the
  textile systems representing induced endomorphisms on the edge
  shift of the given graph, and this allows the powerful tools of the
  theory of textile systems developed by Nasu to be applied to the
  study of permutative endomorphisms.
\end{abstract}

%


\section{Introduction}
The aim of this paper is to introduce a class of labeled directed multigraphs -- called permutation graphs -- which provide a new powerful tool for the study of permutative endomorphisms of graph $C^\ast$-algebras.
This facilitates an intuitive interpretation of the results given in
\cite{chs_endomorphisms_of_graph_algebras} and connects graph $C^\ast$-algebra theory to 
the theory of textile systems developed by Nasu \cite{nasu_memoir} for the
investigation of shift space endomorphisms and automorphisms.

In \cite{cuntz_automorphisms_of_certain}, Cuntz introduced the Weyl
group of the simple purely infinite $C^\ast$-algebras $\OO_n$. It 
arises as the quotient of the normalizer of a maximal Abelian subgroup
of the $C^\ast$-algebra automorphism group. An important
subgroup of the Weyl group corresponds to those automorphisms that
globally preserve the canonical UHF-subalgebra of $\OO_n$. Cuntz
raised the question of how one determines the structure of this subgroup, and
this was answered in Reference \cite{chs_the_restricted_weyl}. In
\cite{chs_endomorphisms_of_graph_algebras}, this program was taken one
step further by expanding it to a much wider class of graph
$C^\ast$-algebras. In this paper, it is shown how permutative
endomorphisms can be represented by labeled directed multigraphs, and
it is shown how properties of such a permutation graph can be used to
determine whether the corresponding permutative endomorphism is an automorphism.

Permutative endomorphisms of graph $C^*$-algebras and especially of
the Cuntz algebras ${\mathcal O}_n$ have already received considerable
attention and have been studied in several different contexts. In
particular, localized endomorphisms (a class of endomorphisms that
includes permutative ones) of the Cuntz algebras were investigated in
the framework of the Jones index theory in
\cite{Izu93,Lon94,CP96,CS09,CRS10,Hay13}. Similar investigations of
localized endomorphisms of the Cuntz-Krieger algebras were carried out
in \cite{Izu98}. Voiculescu's entropy and related properties of
permutative endomorphisms of the Cuntz algebras were studied in
\cite{SZ08,Ska11}. Such endomorphisms play a role in the approach to
wavelet theory via representations of the Cuntz algebras taken in
\cite{BJ99}. An intriguing connection between permutative
automorphisms of the Cuntz algebra ${\mathcal O}_n$ and automorphisms
of the full two-sided $n$-shift was found in
\cite{chs_the_restricted_weyl}. An interesting combinatorial approach
to permutative endomorphisms of ${\mathcal O}_n$ was presented in
\cite{Kaw05}.

In Section \ref{sec_background}, background information, definitions,
and terminology are given, followed by a definition of permutation
graphs and an examination of their properties in Section \ref{sec_pg}.
Here, a new proof is given of the result from
\cite{chs_endomorphisms_of_graph_algebras}, yielding automorphism
conditions for permutative endomorphisms that are formulated directly 
in terms of their permutation graphs. A connection is shown between
permutation graphs and Nasu's textile systems, allowing the computational
methods of textile systems to be used in the analysis of permutative endomorphisms.
In Section \ref{sec_Etau_alg}, an algorithm is presented that serves to 
exhaustively construct all permutation graphs corresponding to 
permutative automorphisms.
In Section \ref{sec_inner_order}, inner equivalence of permutative 
automorphisms is linked to the behavior of the corresponding permutation graphs as dynamical systems. An order of the vertices and edges of a permutation graph is introduced, and used to construct representations of classes of automorphisms that are inner equivalent through permutative unitaries. These representations are textile systems, allowing the machinery developed in \cite{nasu_memoir} to be applied to computations involving such equivalence classes. 
In Section \ref{sec_Otau_alg}, an algorithm is presented that allows efficient exhaustive construction of inner equivalence classes of permutative automorphisms. 
The methods of Sections \ref{sec_Etau_alg}--\ref{sec_Otau_alg} have been implemented in a set of 
computer programs, which are applied to a particular
graph $C^\ast$-algebra in Section \ref{sec_bowtie} and to certain Cuntz algebras in Section \ref{sec:On}. The latter application confirms the results of \cite{conti_szymanski_labeled_trees} and serves to illustrate how the presented methods outperform previously used algorithms.  

The connection to the theory of textile systems, which this paper
opens, paves the way for an efficient search for permutative
automorphisms using computer programs. First of all, the algorithms
presented in Section \ref{sec_Etau_alg} and \ref{sec_Otau_alg} drastically reduce the number of
cases one needs to examine in an exhaustive search for (equivalence classes of) permutative
automorphisms. Secondly, the conditions introduced in Section
\ref{sec_pg} -- guaranteeing that a permutative endomorphism is an
automorphism -- can be efficiently tested using an algorithm that will
be presented in a forthcoming paper. Finally, in collaboration with Brendan Berg,
the first and second authors have developed a number of practical 
algorithms for computations on textile systems that allow one to
efficiently investigate the order of a given permutative automorphism. A
forthcoming paper will apply these techniques to an investigation of
concrete permutative automorphisms like the ones identified in Section
\ref{sec_bowtie}.

Although we study endomorphisms of graph $C^*$-algebras, which are
analytic in nature, our approach relies only on discrete and combinatorial
properties. Therefore we believe that the results of this paper after
small modifications may be applicable to purely algebraic objects. In
particular, Proposition \ref{prop_action_on_path} and Theorem \ref{thm_automorphism} should apply
to permutative endomorphisms of Leavitt path algebras,
\cite{AAP05}. Similarly, adaptations of the algorithms constructed in sections \ref{sec_Etau_alg} and \ref{sec_Otau_alg} can be used to construct permutative automorphisms of Leavitt path algebras. Furthermore, the action of permutative endomorphisms on
the graph $C^*$-algebra restricts to the action on the subgroup of
permutative unitaries inside the unitary group. This restriction gives
rise to interesting endomorphisms and automorphisms of certain locally
finite groups (cf. comments in \cite[Section 3.3]{CS12} about
permutative automorphisms of the infinite symmetric group with
diagonal embeddings). A more detailed account of this aspect of the
theory of endomorphisms of graph algebras will be given elsewhere.

\subsubsection*{Acknowledgments.}
Supported by \textsc{VILLUM FONDEN} through the experimental mathematics network at the University of Copenhagen. Supported by the Danish National Research Foundation through the Centre for Symmetry and Deformation (DNRF92) and by the Danish Natural Science Research Council (FNU).
  The authors are grateful to Mike Boyle and Brendan Berg for valuable
  comments and suggestions concerning the theory of textile systems.

\section{Background and notation}\label{sec_background}

\subsection{Graphs, graph algebras, and permutative endomorphisms.}
Let $E = (E^0,E^1,\range,\source)$ be a \emph{directed multigraph},  
where $E^0$ and
$E^1$ are countable sets of vertices and edges, respectively, while
$\source,\range \colon E^1 \to E^0$ denote the source and range maps.
For notational convenience -- and in accordance with the tradition in the graph algebra literature -- we will deviate from the graph theoretical terminology and simply call these objects \emph{graphs} in the following.
When multiple graphs are considered at the same time, 
the notation $\source_E$ and $\range_E$ will be used to disambiguate to which graph the maps belong.
If an edge $e\in E^1$ has $\source(e) = u$ and $\range(e) = v$, we write $e\colon u\to v$.
A \emph{path} $\mu$ of \emph{length} $l$, is a sequence
of edges $\mu_1 \cdots \mu_l$ for which $\range(\mu_i) = \source(\mu_{i+1})$. 
If $\source(\mu_i) = v_i$ and $\range(\mu_i) = v_{i+1}$ for vertices $v_1,\ldots,v_{l+1}$,
we write $\mu\colon v_1\arc \cdots \arc v_{l+1}$.

For each $l \in \N$, $E^l$ will denote the set of paths in $E$ of
length $l$, and $E^\ast$ will denote the set of all finite paths. The
range and source maps are extended to $E^\ast$ in the natural way.
For $u,v \in E^0$, let $E^l_{u \to v} = \{ \alpha \in E^l \mid
\source(\alpha) = u, \range(\alpha) = v\}$, $E^l_{u \to \ast} = \{ \alpha \in E^l
\mid \source(\alpha) = u\}$, and $E^l_{\ast \to v} = \{ \alpha \in E^l \mid
\range(\alpha) = v\}$. A \emph{sink} is a vertex that emits no edges, and a
\emph{source} is a vertex that receives no edges. A \emph{cycle} is a
path $\mu$ with $\range(\mu) = \source(\mu)$, and a \emph{loop} is a cycle of
length 1. A cycle $\mu \in E^k$ is said to have an \emph{exit} if
there exits $1 \leq i \leq k$ such that $\source(\mu_i)$ emits at least two
edges.

The following definition of the graph $C^*$-algebra corresponding to
an arbitrary countable graph was given in
\cite{fowler_laca_raeburn}. Graph algebras provide a natural
generalization of the Cuntz-Krieger algebras, \cite{cuntz_krieger},
and are a subject of wide-spread investigations by specialists in
the theory of operator algebras, symbolic dynamics, non-commutative
geometry and quantum groups.  The \emph{graph $C^\ast$-algebra}
\cite{kumjian_pask_raeburn_renault,kumjian_pask_raeburn,fowler_laca_raeburn}
of $E$, denoted $C^\ast(E)$, is the universal $C^\ast$-algebra
generated by a collection of mutually orthogonal projections $\{P_v
\mid v \in E^0\}$ and partial isometries $\{S_e \mid e \in E^1\}$
satisfying the relations
\begin{itemize}
\item $S_e^\ast S_e = P_{\range(e)}$ and $S_e^\ast S_f = 0$ when $e \neq f$,
\item $S_e S_e^\ast \leq P_{\source(e)}$,
\item $P_v = \sum_{\source(e) = v} S_e S_e^\ast$ when $v$ emits a non-zero finite number of edges.
\end{itemize}
For a general introduction to graph $C^\ast$-algebras, see
\cite{raeburn}. An \emph{endomorphism} of $C^*(E)$ is a unital $*$-homomorphism from $C^*(E)$ into itself. 
The \emph{Leavitt path algebra} of $E$ is analogously defined 
as the universal algebraic object satisfying the relations
given above. It is not a $\ast$-algebra, however, so a partial
isometry $S_e^\ast$ must be added to the list of generators for each
$e \in E^1$.

For $\mu \in E^k$, let $S_\mu = S_{\mu_1} \cdots S_{\mu_k}$ be the
corresponding non-zero partial isometry in $C^\ast(E)$. The final projection of $S_\mu$ is $P_\mu = S_\mu S_\mu^\ast$, and the initial
projection is $P_{\range(\mu)}$. The final projections in $\{ P_\mu \mid
\mu \in E^\ast \}$ commute, and the $C^\ast$-algebra generated by this
set is called the \emph{diagonal subalgebra} and it is denoted
$\DD_E$. If every cycle in $E$ has an exit, then $\DD_E$ is a maximal
Abelian subalgebra (MASA) in $C^\ast(E)$ as shown in
\cite[Thm.~5.2]{hopenwasser_justin_power} and
\cite[Thm.~3.7]{nagy_reznikoff}.
For simplicity, all graphs considered
in the following will be assumed to be finite, to have no sinks or
sources and to have an exit from every cycle.

Given $k \in \N$, a permutation $\tau \in \Perm (E^k)$ is said to be
\emph{endpoint-fixing}
if $\tau(E^k_{u \to v}) = E^k_{u \to v}$ for
all $u,v \in E^0$. If $\tau \in \Perm (E^k)$ is endpoint-fixing, then
$U_\tau = \sum_{\alpha \in E^k} S_{\tau(\alpha)} S^\ast_\alpha$
defines a unitary in $C^\ast(E)$ for which $U_\tau S_\alpha =
S_{\tau(\alpha)}$ for all $\alpha \in E^k$. The universality of
$C^\ast(E)$ guarantees that an endomorphism $\lambda_\tau \colon
C^\ast(E) \to C^\ast(E)$ can be defined by $\lambda_\tau(S_e) = U_\tau
S_e$. Such an endomorphism is said to be a \emph{permutative
  endomorphism} at \emph{level} $k$. Note that this part of the 
construction is not limited to unitaries arising from permutations: it
can be carried out for any unitary in the multiplier algebra that commutes
with the vertex projections. This is examined in
\cite{chs_endomorphisms_of_graph_algebras}. By the gauge invariant uniqueness theorem, a permutative endomorphism is automatically injective \cite[Prop. 2.1]{chs_endomorphisms_of_graph_algebras}. 

\subsection{Labeled graphs and shift spaces.}
Given a finite set $\AA$, a \emph{labeled} graph with alphabet $\AA$ is a pair
$(E, \LL)$ consisting of a graph $E$ and a surjective \emph{labeling 
  map} $\LL \colon E^1 \to \AA$.  The labeling map naturally extends 
to $E^\ast$.                                                         
A labeled graph $(E, \LL)$ is said to be
\emph{left-resolving} if for all $e_1,e_2 \in E^1$, $\range(e_1) = \range(e_2)$ and $\LL(e_1) = \LL(e_2)$ implies $e_1 = e_2$.
A labeled graph $(E, \LL)$ is said to be
\emph{right-resolving} if for all $e_1,e_2 \in E^1$, $\source(e_1) = \source(e_2)$ and $\LL(e_1) = \LL(e_2)$ implies $e_1 = e_2$.
A labeled graph $(E, \LL)$ is said to be
\emph{left-synchronizing with delay $m \in \N$} if $\source(\alpha) = \source(\beta)$ whenever $\alpha, \beta \in E^m$ and
$\LL(\alpha) = \LL(\beta)$. In the case where no two parallel edges have the same label, the labeled graph is left-synchronizing if and only if it is \emph{right-closing} (see \cite[def. 5.1.4]{lind_marcus} for a definition of right-closing labeled graphs).
Analogously, $(E, \LL)$ is said to be
\emph{right-synchronizing} if there exists $m$ such that any two paths
with the same label and length greater than or equal to $m$ must have the same range.

If $e\in E^1$ has $\source(e) = u$, $\range(e) = v$, and $\LL(e) = a$, we write
$e\colon u\larc{a} v$. Similarly, the statement that $\mu\colon v_1\arc \cdots\arc v_n$ is a 
path over $E$ with labels $\LL(\mu_i) = a_i$ may be written as
$
  \mu\colon v_1 \larc{a_1} \cdots \larc{a_{n-1}} v_n
$.

For a graph $E$, the collection of one-sided infinite paths $\X_E^+ = \{ e_1 e_2 \ldots
\mid e_i \in E^1, \range(e_i) = \source(e_{i+1}) \}$ is a \emph{one-sided shift
  space} equipped with the shift map $\sigma \colon \X_E^+ \to \X_E^+$
defined by $\sigma(e_1 e_2 e_3 \cdots) = e_2 e_3 \cdots$. Similarly,
the collection of bi-infinite paths $\X_E = \{ \cdots e_1 e_2 \cdots
\mid e_i \in E^1, \range(e_i) = \source(e_i+1) \}$ is a \emph{two-sided shift
  space} equipped with the shift map $\sigma \colon \X_E \to \X_E$
defined by $\sigma(x)_i = x_{i+1}$. $X_E$ is called the \emph{edge
  shift} of $E$. For a thorough introduction to the theory of shift
spaces, see \cite{lind_marcus}. The shift space \emph{presented} by a labeled graph $(E,\LL)$ is $\X_{(E,\LL)} = \{ \cdots \LL(e_1) \LL(e_2) \cdots
\mid e_i \in E^1, \range(e_i) = \source(e_{i+1}) \}$. The \emph{language} of such a shift space is $\LL(E^\ast)$. The \emph{sliding block code induced by the labeling} is the map $\phi_\LL \colon \X_E \to \X_{(E,\LL)}$ defined by $\phi_\LL(\cdots e_1 e_2 \cdots) = \cdots \LL(e_1) \LL(e_2) \cdots$.

It is well-known (see e.g.\ \cite[Thm.\ 3.7]{webster}) that the
Gelfand spectrum of $\DD_E$ can be identified with $\X^+_E$ via the
identification of $z^+ \in \X^+_E$ with the map $\phi_{z^+} \colon
\DD_E \to \C$ defined by
\begin{displaymath}
\phi_{z^+}(P_\mu) = \left\{  
\begin{array}{c l}
1, & \mu \textnormal{ is a prefix of } z \\
0, & \textrm{otherwise}
\end{array}
\right. .
\end{displaymath}
Hence, an endomorphism of $C^\ast(E)$ that preserves $\DD_E$ will
induce an endomorphism of $\X^+_E$.

\section{Permutation graphs}\label{sec_pg}
\subsection{Definition and basic properties.}
Let $E$ be a finite graph without sinks or sources where every cycle has an exit.

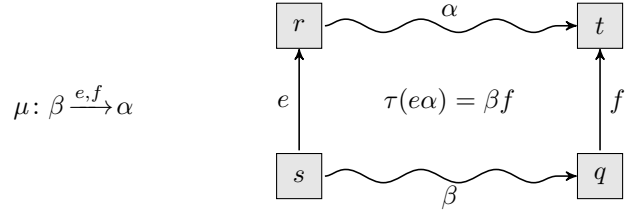
\begin{figure}
\begin{center}
\begin{tikzpicture}
  [bend angle=10,
   clearRound/.style = {circle, inner sep = 0pt, minimum size = 17mm},
   clear/.style = {rectangle, minimum width = 5 mm, minimum height = 5 mm, inner sep = 0pt},  
   greyRound/.style = {circle, draw, minimum size = 1 mm, inner sep =
      0pt, fill=black!10},
   grey/.style = {rectangle, draw, minimum size = 6 mm, inner sep =
      1pt, fill=black!10},
    white/.style = {rectangle, draw, minimum size = 6 mm, inner sep =
      1pt},
   to/.style = {->, shorten <= 1 pt, >=stealth', semithick}]
  
  \node[grey] (se) at (0,0) {$s$};
  \node[grey] (re) at (0,2) {$r$};

  \node[grey] (sf) at (4,0) {$q$};
  \node[grey] (rf) at (4,2) {$t$};

  \node[clear] (edge) at (2,1) {$\tau(e \alpha) = \beta f$};
  \node[clear] (edge) at (-3,1) {$\mu \colon \beta \llarc{e}{f} \alpha$};

  \draw[to] (se) to node[auto] {$e$} (re); 
  \draw[to] (sf) to node[auto,swap] {$f$} (rf); 
  \draw[to,snake it] (se) to node[auto,swap] {$\beta$} (sf);
  \draw[to,snake it] (re) to node[auto] {$\alpha$} (rf);   
\end{tikzpicture}
\end{center}
\caption{Illustration of the relationship between paths over $E$ and labeled edges in $E_\tau$.  
For each $e\in E^1$ and $\alpha\in E^{k-1}$ for which $\range_E(e) = \source_E(\alpha)$,
the relation $\tau(e\alpha) = \beta f$ is encoded in $E_\tau$ as the labeled edge $\mu\colon\beta\llarc{e}{f}\alpha$.
The respective sources and ranges in $E$ must be of the form
$e\colon s\to r$, $\alpha\in E^{k-1}_{r\to t}$, $\beta\in E^{k-1}_{s\to q}$, and $f\colon q\to t$
for appropriate $s,r,q,t\in E^0$,
as shown in the diagram on the right.
}
\label{fig_Etau}
\end{figure}

\begin{definition}\label{def_Etau} 
Let $k \in \N$ and let $\tau \in \Perm (E^k)$
be an endpoint-fixing permutation. Define a graph $E_\tau = (E_\tau^0,E_\tau^1,\source_\tau,\range_\tau)$ by
\begin{displaymath}
\begin{array}{l c l}
E_\tau^0 = E^{k-1}  & \qquad &
\source_\tau(\mu) = \tau(\mu)_1 \cdots \tau(\mu)_{k-1} \\
E_\tau^1 = E^{k}  & &
\range_\tau(\mu) = \mu_2 \cdots \mu_{k}
\end{array}.
\end{displaymath}
Define $\LL_1, \LL_2 \colon E_\tau^1 \to E^1$ by
\begin{displaymath}
\LL_1(\mu) = \mu_1
\qquad \textnormal{ and } \qquad 
\LL_2(\mu) = \tau(\mu)_{k},
\end{displaymath}
and define a labeling $\LL_\tau \colon E_\tau^1 \to E^1 \times E^1$ by $\LL_\tau(\mu) = [\LL_1(\mu),\LL_2(\mu)]$. The \emph{permutation graph} of $E$ and $\tau$ is defined as the labeled graph $(E_\tau,\LL_\tau)$. This will be said to be a permutation graph of $E$ at \emph{level} $k$.
\end{definition}

A permutation is exactly determined by the corresponding permutation graph, since the permutation graph $(E_\tau,\LL_\tau)$ contains an edge  
labeled $[e,f]$ from $\beta \in E^{k-1}$ to $\alpha \in E^{k-1}$ if and only if
$\tau(e\alpha) = (\beta f)$. This is illustrated in Figure \ref{fig_Etau}.
The permutation graph defined above is introduced as a key tool 
for investigations of permutative endomorphisms. It constitutes a significant improvement to the graphical devices 
constructed ad hoc in \cite{conti_szymanski_labeled_trees} and \cite{chs_endomorphisms_of_graph_algebras}.


\begin{lem}\label{lem_Etau_properties}
Let $E$ and $\tau$ be as in Definition \ref{def_Etau} above.
\begin{enumerate}
\item 
Let $e \in E^1$. For each edge $\mu\colon \alpha \llarc{e}{f} \beta$ in $E_\tau$,  $\source_E(\beta) = \source_E(e)$ and $\source_E(\alpha) = \range_E(e)$. \label{lem_Etau_properties_L1}
\item For every edge $e\colon s\to r$ in $E^1$ and $\alpha\in E^{k-1}_{r\to\ast}$ there is a unique 
  edge $\mu\colon \beta\llarc{e}{f}\alpha$ of $(E_\tau,\LL_\tau)$.
\label{lem_Etau_properties_receive}
\item Let $f \in E^1$. For each edge $\mu\colon \beta \llarc{e}{f} \alpha$ in $E_\tau$, $\range_E(\beta) = \source_E(f)$ and $\range_E(\alpha) = \range_E(f)$. \label{lem_Etau_properties_L2}
\item For every edge $f\colon q\to t$ in $E^1$ and $\beta\in E^{k-1}_{\ast\to q}$ there is a unique   
edge $\mu\colon \beta\llarc{e}{f}\alpha$ of $(E_\tau,\LL_\tau)$.
\label{lem_Etau_properties_emit}  
\end{enumerate}
\end{lem}

\begin{proof}
The statements \enumref{lem_Etau_properties_L1} and \enumref{lem_Etau_properties_L2} follow from the requirement that $\tau$ is endpoint-fixing as seen in Figure \ref{fig_Etau}. 
Bijectivity of $\tau$ entails that $e,\alpha$ uniquely determines $\beta,f$, and vice versa, whereby \enumref{lem_Etau_properties_receive} and \enumref{lem_Etau_properties_emit}
follow.
%
\end{proof}

\noindent The following corollary is an immediate consequence of Lemma \ref{lem_Etau_properties} \enumref{lem_Etau_properties_L1} and \enumref{lem_Etau_properties_L2}. 

\begin{cor}\label{cor_Etau_resolving}
Let $E$ and $\tau$ be as in Definition \ref{def_Etau} above. Then $(E_\tau,\LL_1)$ is left-resolving and $(E_\tau,\LL_2)$ is right-resolving.
\end{cor}


In \cite{conti_szymanski_labeled_trees}, similar conditions were
developed to deal specifically with the permutative endomorphisms of
the Cuntz-algebras $\OO_n$. In that case, the conditions were phrased
in terms of certain labeled trees carrying the same information as the
permutation graph. It was shown that it was sufficient to consider
certain configurations of these labeled trees when searching for
permutative endomorphisms, and the conditions given above generalize
those results.  It is straightforward to check that the properties
given in Lemma \ref{lem_Etau_properties} are also sufficient to
characterize permutation graphs:

\begin{prop}\label{prop_sufficient}
  Let $(F,\LL)$ be a labeled graph with vertex set $F^0 = E^{k-1}$ and labels $\LL(F^1) \subseteq E^1 \times E^1$. Then $(F,\LL)$ is the permutation graph of some
  endpoint-fixing permutation of $E^k$ if and only if $(F,\LL)$
  satisfies the conditions given in Lemma \ref{lem_Etau_properties}.
\end{prop}

%

The following proposition shows that $(E_\tau,\LL_\tau)$ gives two
alternate presentations of the edge shift presented by $E$.

\begin{prop}\label{prop_same_shift}
  The two-sided shift spaces $\X_E$, $\X_{(E_\tau, \LL_1)}$, and
  $\X_{(E_\tau, \LL_2)}$ presented by, respectively, $E$, $(E_\tau,
  \LL_1)$, and $(E_\tau, \LL_2)$ are identical.
\end{prop}

\begin{proof}
It is sufficient to prove that the three labeled graphs have the same language.
As any other labeling map, $\LL_1$ can be extended naturally and considered as a map from $E_\tau^\ast$ to $E^\ast$, and using this, the language of $\X_{(E_\tau, \LL_1)}$ is $\LL_1(E_\tau^\ast)$.
By Lemma \ref{lem_Etau_properties} \enumref{lem_Etau_properties_receive}, $\LL_1(E_\tau^1) = E^1$. Let $n \geq 1$ be given and assume that $\LL_1(E_\tau^n) = E^n$. Let $e \in E^0$. By Lemma \ref{lem_Etau_properties} \enumref{lem_Etau_properties_L1}  and \enumref{lem_Etau_properties_receive}, $ew \in E^{n+1}$ if and  only if $ew \in \LL_1(E_\tau^{n+1})$.   
  By induction, $\LL_1(E_\tau^\ast) = E^\ast$, and an analogous argument proves that $\LL_2(E_\tau^\ast) = E^\ast$.
%
\end{proof}

\noindent The presentations $(E_\tau,
  \LL_1)$, and $(E_\tau, \LL_2)$ do not in general correspond to any
standard presentation of $\X_E$ known to the authors. However, if the
permutation is the identity, then the permutation graph is simply a higher
block presentation of the original edge shift (see e.g.\ \cite{lind_marcus}).

\begin{example} \label{ex_golden_mean}
Consider the graph $E$ shown in Figure \ref{fig_golden_mean}, and let
\begin{displaymath}
\tau = (111 , 132, 321)(113, 323) \in \Perm(E^3).
\end{displaymath}
This is the endpoint-fixing permutation considered in \cite[Example
6.5]{chs_endomorphisms_of_graph_algebras}. Its permutation graph
$(E_\tau,\LL_\tau)$ is shown in Figure \ref{fig_golden_mean_perm}.
It is easy to check that $E$, $(E_\tau,
\LL_1)$, and $(E_\tau, \LL_2)$ all present the same two-sided shift space,
as required by Proposition \ref{prop_same_shift}.

\begin{figure}
\begin{center}
\begin{tikzpicture}
  [bend angle=10,
   clearRound/.style = {circle, inner sep = 0pt, minimum size = 17mm},
   clear/.style = {rectangle, minimum width = 5 mm, minimum height = 5 mm, inner sep = 0pt},  
   greyRound/.style = {circle, draw, minimum size = 1 mm, inner sep =
      0pt, fill=black!10},
   grey/.style = {rectangle, draw, minimum size = 6 mm, inner sep =
      1pt, fill=black!10},
    white/.style = {rectangle, draw, minimum size = 6 mm, inner sep =
      1pt},
   to/.style = {->, shorten <= 1 pt, >=stealth', semithick}]  
    
  \node[grey] (v1) at (-4,0) {$v_1$};
  \node[grey] (v2) at (-2,0) {$v_2$};

  \draw[to,loop left] (v1) to node[auto] {$1$} (v1);
  \draw[to,bend left] (v1) to node[auto] {$3$} (v2);
  \draw[to,bend left] (v2) to node[auto] {$2$} (v1);
\end{tikzpicture}
\end{center}
\caption{The graph $E$ considered in Example \ref{ex_golden_mean}.} 
\label{fig_golden_mean}
\end{figure}
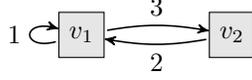

\begin{figure}
\begin{center}
\begin{tikzpicture}
  [bend angle=10,
   clearRound/.style = {circle, inner sep = 0pt, minimum size = 17mm},
   clear/.style = {rectangle, minimum width = 5 mm, minimum height = 5 mm, inner sep = 0pt},  
   greyRound/.style = {circle, draw, minimum size = 1 mm, inner sep =
      0pt, fill=black!10},
   grey/.style = {rectangle, draw, minimum size = 6 mm, inner sep =
      1pt, fill=black!10},
    white/.style = {rectangle, draw, minimum size = 6 mm, inner sep =
      1pt},
   to/.style = {->, shorten <= 1 pt, >=stealth', semithick},
   lbl/.style = {font=\footnotesize,inner sep=0.5mm}
   ]
   
  \node[grey] (11) at (0,1.5) {$11$};
  \node[grey] (32) at (0,0) {$32$};
  \node[grey] (13) at (0,-2) {$13$};
  \node[grey] (21) at (4,1.5) {$21$};
  \node[grey] (23) at (4,-2) {$23$};

  \node[rectangle, draw, dashed, minimum height = 3cm, minimum width = 1.5 cm] (v1) at (0,0.75) {};
  \node[clear] (s11) at (0,2.5) {$E^2_{v_1 \to v_1}$};
    \node[rectangle, draw, dashed, minimum height = 1.5cm, minimum width = 1.5 cm] (v1) at (13) {};
  \node[clear] (s12) at (0,-3) {$E^2_{v_1 \to v_2}$};
 \node[rectangle, draw, dashed, minimum height = 1.5cm, minimum width = 1.5 cm] (v2) at (21) {};
  \node[clear] (s21) at (4,2.5) {$E^2_{v_2 \to v_1}$};
 \node[rectangle, draw, dashed, minimum height = 1.5cm, minimum width = 1.5 cm] (v2) at (21) {};
 \node[clear] (s22) at  (4,-3) {$E^2_{v_2 \to v_2}$};
 \node[rectangle, draw, dashed, minimum height = 1.5cm, minimum width = 1.5 cm] (v2) at (23) {};
 
   
  \draw[to,loop above] (32) to node[lbl,auto] {$[1,1]$} (32);
  \draw[to] (32) to node[lbl,midway,fill=white] {$[1,3]$} (13);
  \draw[to, bend left = 80] (13) to node[lbl,auto] {$[1,2]$} (11);
  
  \draw[to, bend left] (11) to node[lbl,midway,fill=white] {$[3,1]$} (21);
  \draw[to] (11) to node[lbl,near end, fill=white] {$[3,3]$} (23);

  \draw[to, bend left] (21) to node[lbl,midway,fill=white] {$[2,1]$} (11);
  \draw[to] (21) to node[lbl,near start, fill=white] {$[2,3]$} (13);
  \draw[to] (23) to node[lbl,midway,fill=white] {$[2,2]$} (32);
\end{tikzpicture}
\end{center}
\caption{The permutation graph $(E_\tau,\LL_\tau)$ considered in Example \ref{ex_golden_mean}. The dashed boxes contain vertices corresponding to paths in $E^2$ with the specified sources and ranges.} 
\label{fig_golden_mean_perm}
\end{figure}
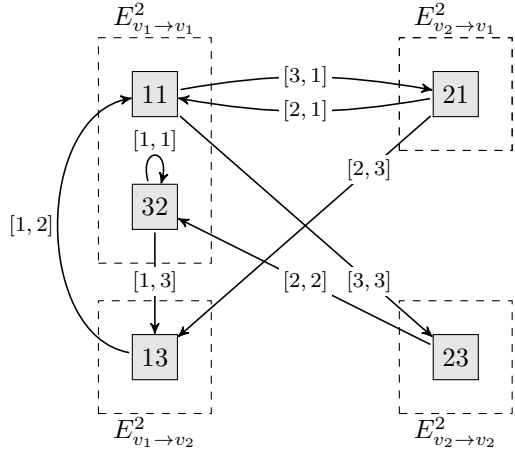
\end{example}

\subsection{Images via permutation graphs.}\label{sec_images}
Now, it will be shown how permutation graphs can be used
in direct computations involving permutative endomorphisms.

\begin{lem}\label{lem_image_Se}%
  Let $\tau \in \Perm(E^k)$ be endpoint-fixing.
  Let $e \in E^1$ and consider $\{\mu_1, \ldots, \mu_n\} = \LL_1^{-1}(e)
  \subseteq E_\tau^1$. For each $i$, let $\beta_i = \source_\tau(\mu_i)$, $\alpha_i = \range_\tau(\mu_i)$,
  and $f_i = \LL_2(\mu_i)$, i.e.\ $\mu_i\colon \beta_i \llarc{e}{f_i} \alpha_i$.
   Then
  \begin{displaymath}
    \lambda_\tau(S_e) = \sum_{i=1}^n S_{\beta_i} S_{f_i} S_{\alpha_i}^\ast. 
  \end{displaymath} 
\end{lem}

\begin{proof}
  By the definition of $\lambda_\tau$ and the construction of the permutation graph 
  \begin{displaymath}
    \lambda_\tau(S_e) 
    = \bigg( \sum_{\gamma \in E^k} S_{\tau(\gamma)} S_\gamma^\ast \bigg) S_e
    =  \sum_{(e \alpha) \in E^k} S_{\tau(e \alpha)} S_\alpha^\ast  
    =  \sum_{i=1}^n S_{\beta_i} S_{f_i} S_{\alpha_i}^\ast.
  \end{displaymath}
\end{proof}

\noindent This allows the permutation graph to be used to perform direct calculations, and additionally
as a visual tool in computations: 
To find the image of $S_e$, one simply needs to identify the edges in the permutation graph for which the first label is $e$ and construct the sum given above.

It is straightforward to extend Lemma \ref{lem_image_Se} to an arbitrary  path $e_1 \cdots e_n$ by applying the result to individual generators $S_{e_i}$ and using the relations of $C^\ast(E)$ to achieve cancellations. This leads to:

\begin{prop}\label{prop_action_on_path} 
  Let $\tau \in \Perm(E^k)$ be endpoint-fixing.
  Let $l \in \N$, and let $\gamma \in E^l$. Let $\{A_1, \ldots, A_n\}
  = \LL_1^{-1}(\gamma) \subseteq E_\tau^l$. For each $i$, let $\beta_i = \source_\tau(A_i)$,
  $\alpha_i = \range_\tau(A_i)$, and $\delta_i = \LL_2(A_i)$. Then
\begin{equation}\label{eq_action_on_path}
  \lambda_\tau(S_\gamma) = \sum_{i=1}^n S_{\beta_i} S_{\delta_i} S_{\alpha_i}^\ast. 
\end{equation} 
\end{prop}

\subsection{Permutation graphs as textile systems.}\label{sec_textile}
A \emph{textile system}, $T = (\Gamma, G, p,q)$, over a graph $G$ consists of a
graph $\Gamma = (\Gamma^0,\Gamma^1,\source_\Gamma,r_\Gamma)$ and a pair of graph homomorphisms $p,q \colon
\Gamma \to G$ such that for each edge $a \in \Gamma^1$, the quadruple
$(\source_\Gamma(a),\range_\Gamma(a),p(a),q(a))$ uniquely determines $a$
\cite[p.~14]{nasu_memoir}. Textile systems were introduced by Nasu to
investigate endomorphisms and automorphisms of shift spaces, and a
powerful set of tools has been developed for this purpose.
A textile system $T = (\Gamma, G, p,q)$ is said to be \emph{in the standard form} if $\Gamma$ and $G$ have no sinks or sources and the graph-homomorphisms $p,q$ are onto \cite[p.~18]{nasu_memoir}. From now on, all textile systems will be assumed to be in the standard form. 

For a textile system $T = (\Gamma, G, p, q)$, define maps $\phi_1, \phi_2 \colon \X_\Gamma \to \X_G$ by
\begin{align*}
\phi_1(\ldots a_{-1} a_0 a_1 \ldots ) &= \ldots p(a_{-1}) p(a_{0}) p(a_{1}) \ldots \\
\phi_2(\ldots a_{-1} a_0 a_1 \ldots ) &= \ldots q(a_{-1}) q(a_{0}) q(a_{1}) \ldots  
\end{align*}
Define $X_0 = \X_G$ and $Z_0 = \X_\Gamma$. For each $l \in \N$, define 
\begin{align*}
& X_l = \phi_1(Z_{l-1}) \cap \phi_2(Z_{l-1})  \\
& Z_l = \phi^{-1}_1(X_{l}) \cap \phi_2^{-1}(X_{l}),
\end{align*}
and let
\begin{displaymath}
\X_T = \bigcap_{l=0}^\infty X_k 
\quad  \textnormal{and}  \quad
\ZZ_T = \bigcap_{l=0}^\infty Z_k.
\end{displaymath}
The textile system $T$ is said to be \emph{non-degenerate} if $\X_T = \X_\Gamma$ (from which it follows that $\ZZ_T = \X_G$) \cite[p.~19]{nasu_memoir}. For a non-degenerate textile system, the maps $\phi_1,\phi_2 \colon \X_\Gamma \to \X_G$ are surjective. In all cases, the restrictions $\phi_1\vert_{\ZZ_T },\phi_2\vert_{\ZZ_T } \colon \ZZ_T \to \X_T$ have image $\X_T$.

A non-degenerate textile system, $T$, is said to be \emph{one-sided 1--1} if $\phi_1$ is injective and \emph{(two-sided) 1--1} if both $\phi_1$ and $\phi_2$ are injective \cite[p.~19]{nasu_memoir}.  If $T_\tau$ is one-sided 1--1, then the map $\phi_{T} = \phi_2 \circ \phi_1^{-1} \colon \X_G \to \X_G$ is an endomorphism of shift spaces, and this endomorphism is said to be \emph{coded by} $T$. If $T$ is two-sided 1--1, then this is an automorphism of shift spaces.


A graph homomorphism $h \colon \Gamma \to G$ is said to be \emph{right-resolving} if for each $\beta \in \Gamma^0$ and $e \in G^1$ with $\source_G(e) = h(\beta)$ there exists a unique $\mu \in \Gamma^1$ such that $\source_\Gamma(\mu) = \beta$ and $h(\mu) = e$ \cite[p.~40]{nasu_memoir}. Note that $(\Gamma, h)$ is a right-resolving labeled graph with alphabet $G^1$ when $h$ is a right-resolving graph homomorphism (but the converse is not always true). Left-resolving graph homomorphisms are defined analogously. A textile system is said to be \emph{LR} if $p$ is a left-resolving graph homomorphism and $q$ is a right-resolving graph homomorphism \cite[p.~40]{nasu_memoir}. It is straightforward to check that an LR textile systems is always non-degenerate \cite[Fact 3.3]{nasu_memoir}.


Given a permutation graph $(E_\tau,\LL_\tau)$ over a graph $E$, define maps $p_\tau,q_\tau \colon E_\tau \to E$ by 
\begin{align*}
p_\tau(\mu) &= \LL_1(\mu)  \qquad p_\tau(\beta) = \source_E(\beta) \\
q_\tau(\mu) &= \LL_2(\mu)  \qquad q_\tau(\beta) = \range_E(\beta)
\end{align*}
for all $\mu \in E_\tau^1$ and $\beta \in E_\tau^0 = E^{k-1}$, and let $T_\tau = (E_\tau, E, p_\tau,q_\tau)$. Lemma \ref{lem_Etau_properties} implies that $p_\tau, q_\tau$ are graph homomorphisms and that $(\source_\tau(\mu),\range_\tau(\mu),p_\tau(\mu),q_\tau(\mu))$ uniquely determines each $\mu \in E_\tau^1$, so $T_\tau$ is a textile system over $E$. Furthermore, Lemma \ref{lem_Etau_properties} implies that this textile system is LR, and hence, non-degenerate. The connection between the shift space endomorphism $\phi_{T_\tau}$ 
encoded by $T_\tau$ when $T_\tau$ is one-sided 1--1 and the endomorphism $\lambda_\tau$ of
$C^\ast(E)$ will be explored in the following sections.

\subsection{Automorphism criteria.}\label{sec_conditions}
Let $\lambda_\tau$ be a permutative endomorphism of $C^\ast(E)$ at
level $k$. Combinatorial conditions for the invertibility of
$\lambda_\tau$ have been developed in
\cite{chs_endomorphisms_of_graph_algebras}. These conditions can be
reformulated in terms of the permutation graph, but here an
independent formulation is given together with a new proof of the
most interesting implication.

\begin{remark}\label{rem_chs_graph}
To make the connection to \cite{chs_endomorphisms_of_graph_algebras} explicit, this remark gives a short discussion of the terminology used in that paper.
In the notation of \cite{chs_endomorphisms_of_graph_algebras}, define for
  each $e \in E^1$ the map $f_e \colon E^{k-1}_{\range(e) \to \ast} \to
  E^{k-1}_{\source(e) \to \ast}$ by $f_e(\alpha) = \source_\tau(e\alpha)$. The
  graph $(E_\tau,\LL_1)$ has an edge labeled $e$ from $\beta \in
  E^{k-1}$ to $\alpha \in E^{k-1}$ if and only if $f_e(\alpha) =
  \beta$. Hence, reversing the edges on $(E_\tau,\LL_1)$ will yield
  the graph considered in 
  \cite{chs_endomorphisms_of_graph_algebras}.
As an example of this, notice how the arrows of the permutation graph considered  in Example \ref{ex_golden_mean} have been reversed compared to the corresponding graph from \cite[Example
6.5]{chs_endomorphisms_of_graph_algebras}.
\end{remark}

The following lemma will help connect the theory of permutative
endomorphisms to the theory of textile systems. An analogous result
holds for right-resolving graphs. This is a standard result in
symbolic dynamics, so the proof is omitted. 

\begin{lem}\label{lem_synch}
Let $(G, \LL_1)$ be a finite left-resolving labeled graph. The following are equivalent:
\begin{enumerate}
\item The one-block code induced by the labeling of $(G, \LL_1)$ is invertible.
\item $(G, \LL_1)$ is left-synchronizing.
\item $(G, \LL_1)$ admits no two distinct cycles with the same label. 
\end{enumerate}
\end{lem}

%
%

A permutation graph $(E_\tau,\LL_\tau)$ is said to be \emph{synchronizing in the first
  label} if $(E_\tau,\LL_1)$ is left-synchronizing. Similarly, a
permutation graph is said to be \emph{synchronizing in the second
  label} if $(E_\tau,\LL_2)$ is right-synchronizing. In
\cite{chs_endomorphisms_of_graph_algebras}, equivalent conditions were
denoted property (b) and property (d), respectively.

\begin{lem}\label{lem_labeltosr}
If a permutation graph $(E_\tau,\LL_\tau)$ is synchronizing in the first or second label then there exists $n \in \N$ such that for each $A \in E_\tau^n$, $\LL_\tau(A)$ uniquely determines $\source_\tau(A)$ and $\range_\tau(A)$.
\end{lem}

\begin{proof}
Assume that $(E_\tau,\LL_\tau)$ is synchronizing in the first label and choose $n \in  \N$ such that for each $A \in E_\tau^n$, $\LL_1(A)$ uniquely determines $\source_\tau(A)$. Since $(E_\tau,\LL_2)$ is right-resolving, $\range_\tau(A)$ is uniquely determined by $\source_\tau(A)$ and $\LL_2(A)$. The other part of the statement is shown analogously.
\end{proof}

\begin{lem}\label{lem_Pmu}
Let $(E_\tau,\LL_\tau)$ be a permutation graph. The following are equivalent:
\begin{enumerate}
\item The labeled graph $(E_\tau,\LL_\tau)$ is synchronizing in the first label.\label{lem_Pmu_synch}
\item $T_\tau$ is one-sided 1--1.\label{lem_Pmu_1-1}
\item The endomorphism $\lambda_\tau \colon C^\ast(E) \to C^\ast(E)$ restricts to an automorphism of $\DD_E$.\label{lem_Pmu_auto}
\end{enumerate}
\end{lem}

\begin{figure}
\begin{center}
\begin{tikzpicture}
  [bend angle=10,
   clearRound/.style = {circle, inner sep = 0pt, minimum size = 17mm},
   clear/.style = {rectangle, minimum width = 5 mm, minimum height = 5 mm, inner sep = 0pt},  
   greyRound/.style = {circle, draw, minimum size = 1 mm, inner sep =
      0pt, fill=black!10},
   grey/.style = {rectangle, draw, minimum size = 6 mm, inner sep =
      1pt, fill=black!10},
    white/.style = {rectangle, draw, minimum size = 6 mm, inner sep =
      1pt},
   to/.style = {->, shorten <= 1 pt, >=stealth', semithick}]
  
  \node[grey] (start) at (0,0) {$\mu_{1\ldots k-1}$};
  \node[grey] (alpha) at (3,0) {$\alpha$};

  \node[grey] (betaij) at (6,1) {$\beta_{ij}$};
  \node[grey] (betain) at (6,-1) {$\beta_{ij'}$};

  \draw[to] (start) .. controls (1,-1) and (2, 1) .. node[near end, above] {$[\nu, \mu_{k\ldots\vert \mu \vert}]$} (alpha); 
    \draw[to] (alpha) .. controls (4,2) and (5, 0) .. node[near start, above] {$[\gamma_i, \eta_{ij}]$} (betaij); 
    \draw[to] (alpha) .. controls (4,-2) and (5, 0) .. node[near start, below] {$[\gamma_i, \eta_{ij'}]$} (betain); 
\end{tikzpicture}
\end{center}
\caption{Construction of the paths considered in the proof of Lemma \ref{lem_Pmu}.}
\label{fig_Pmu}
\end{figure}
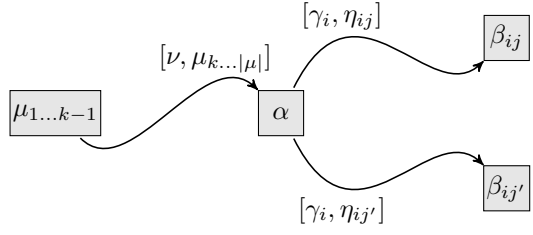

\begin{proof}
  $\enumref{lem_Pmu_synch} \Longleftrightarrow \enumref{lem_Pmu_1-1}$: By Lemma \ref{lem_synch}, $(E_\tau,
  \LL_1)$ is left-synchronizing if and only if the sliding block code
  induced by the labeling is invertible.

  $\enumref{lem_Pmu_synch} \implies \enumref{lem_Pmu_auto}$: Let $\tau$ be an endpoint-fixing permutation
  of $E^k$, and assume that $(E_\tau,\LL_1)$ is left-synchronizing 
  with delay $l$. 
 The permutative endomorphism $\lambda_\tau$ is automatically injective \cite[Prop. 2.1]{chs_endomorphisms_of_graph_algebras}, so it is sufficient to prove surjectivity.
Let $\mu \in E^\ast$ with $\vert \mu \vert \geq
  k$ be given. The aim of the following is to use Proposition \ref{prop_action_on_path} to construct an element $x \in \DD_E$ for which $\lambda_\tau(x) = P_\mu$. This process is illustrated in Figure
  \ref{fig_Pmu}.
By Lemma \ref{lem_Etau_properties}, the vertex $\mu_{1 \ldots k-1} \in E_\tau^0 =
  E^{k-1}$ emits a unique path with second label 
  $\mu_{k \ldots \vert \mu \vert}$. Let $\nu$ be the first label of this path (i.e.\ its image under $\LL_1$)
  and let $\alpha \in E_\tau^0 = E^{k-1}$ be the range.

In order to use Proposition \ref{prop_action_on_path}, it is necessary to find paths in the permutation graph with the required first and second labels.
  By Lemma \ref{lem_Etau_properties}, there is a unique path in
  $E^\ast_\tau$ with source $\alpha$ and second label $\eta \in E^\ast$ precisely when $\source_E(\eta) = \range_E(\nu)$.
   Consider
  the paths in $E^*_\tau$ with source $\alpha$ and length $l$. Let $\{
  \gamma_i \}$ be the set of first labels for all such paths. 
  Left-synchronization implies that any path in $E^*_\tau$ with first label $\gamma_i$ must start at $\alpha$. This is illustrated in Figure \ref{fig_Pmu} for a single $\gamma_i$.
  For each $i$, there may be more than one
  path with first label $\gamma_i$, but the second label uniquely
  identifies the path by Lemma \ref{lem_Etau_properties}. 
   Let $\{
  \eta_{ij} \}$ be the set of second labels of paths with first label
  $\gamma_i$, and let $\beta_{ij}$ be the range of the unique path
  with second label $\eta_{ij}$ in $E_\tau$. Because
  $(E_\tau,\LL_1)$ is left-resolving, $\beta_{ij} \neq \beta_{ij'}$
  whenever $j \neq j'$. The left-synchronization and left-resolvancy
  of $(E_\tau,\LL_1)$ implies that any path with first label $\nu
  \gamma_i$ must have $\mu_{k\ldots \vert \mu \vert}$ as a prefix of
  the second label. By Proposition \ref{prop_action_on_path}, this
  means that
\begin{multline*}
\lambda_\tau \left(\sum_i S_\nu S_{\gamma_i} S_{\gamma_i}^\ast S_\nu^\ast \right) \\
= S_{\mu_{1\ldots k-1}} S_{\mu_{k \ldots \vert \mu \vert}} \sum_{i,j,j'} \left( S_{\eta_{ij}} S_{\beta_{ij}}^\ast S_{\beta_{ij'}} S_{\eta_{ij'}}^\ast \right)   
S_{\mu_{k \ldots \vert \mu \vert}}^\ast S_{\mu_{1\ldots k-1}}^\ast \\
= S_\mu \sum_{ij} \left( S_{\eta_{ij}} S_{\eta_{ij}}^\ast \right) S_\mu^\ast
= S_\mu S_\mu^\ast.  
\end{multline*}
As $\mu$ was arbitrary, this shows that $\lambda_\tau$ restricts to a surjection of $\DD_E$.



$\enumref{lem_Pmu_auto} \implies \enumref{lem_Pmu_synch}$: This is a restatement of a result proved in
\cite[Lem.~6.1]{chs_endomorphisms_of_graph_algebras} under the
identification given in Remark \ref{rem_chs_graph}.
\end{proof}

\begin{thm}\label{thm_automorphism}
The following are equivalent:
\begin{enumerate}
\item $(E_\tau,\LL_\tau)$ is synchronizing in both the first and the second label.\label{thm_auto_synch}
\item $T_\tau$ is two-sided 1--1.\label{thm_auto_1-1}
\item $\lambda_\tau$ is an automorphism of $C^\ast(E)$.\label{thm_auto_auto}
\end{enumerate}
Furthermore, such an automorphism has a permutative inverse.
\end{thm}

\begin{proof}
  $\enumref{thm_auto_synch} \Longleftrightarrow \enumref{thm_auto_1-1}$: As in the proof of Lemma
  \ref{lem_Pmu}, this follows from Lemma \ref{lem_synch}.

  $\enumref{thm_auto_synch} \implies \enumref{thm_auto_auto}$: By \cite[Prop. 2.1]{chs_endomorphisms_of_graph_algebras}, 
  $\lambda_\tau$ is injective. To prove surjectivity, choose $l$ such that $(E_\tau,\LL_1)$ is
  left-synchronizing with delay $l$ and $(E_\tau,\LL_2)$ is
  right-synchronizing with delay $l$. Let $\mu, \nu \in E^\ast$ have the same endpoint
  $\range(\mu)=\range(\nu)$, and $\vert \mu \vert, \vert \nu \vert \geq k$. 
  As in the proof of Lemma \ref{lem_Pmu}, the aim is to use Proposition \ref{prop_action_on_path} to choose an element of $C^\ast(E)$ defined by a suitable collection of paths in the permutation graph. More specifically, the aim is to construct an element $x$, such that $\lambda_\tau(x) = S_\mu S_\nu^\ast$. This process is illustrated in Figure  \ref{fig_SmuSnu}. Such elements will then be used to construct a permutative unitary $U_\pi$ for which $\lambda_\pi$ is the inverse of $\lambda_\tau$. 

  Consider the unique path in $E^*_\tau$ with source node $\mu_{1\ldots
    k-1}$ and second label $\mu_{k\ldots\vert \mu \vert}$. Let $\bar
  \mu$ be the first label of this path, and let $\alpha \in E_\tau^0
  = E^{k-1}$
  be the range. Similarly, consider the unique path in
  $E^\ast_\tau$ with source $\nu_{1\ldots k-1}$ and second label
  $\nu_{k\ldots\vert \nu \vert}$. Let $\bar \nu$ be the first label of
  this path, and let $\beta \in E_\tau^0 = E^{k-1}$ be the range.

\begin{figure}
\begin{center}
\begin{tikzpicture}
  [bend angle=10,
   clearRound/.style = {circle, inner sep = 0pt, minimum size = 17mm},
   clear/.style = {rectangle, minimum width = 5 mm, minimum height = 5 mm, inner sep = 0pt},  
   greyRound/.style = {circle, draw, minimum size = 1 mm, inner sep =
      0pt, fill=black!10},
   grey/.style = {rectangle, draw, minimum size = 6 mm, inner sep =
      1pt, fill=black!10},
    white/.style = {rectangle, draw, minimum size = 6 mm, inner sep =
      1pt},
   to/.style = {->, shorten <= 1 pt, >=stealth', semithick}]
  
  \node[grey] (mu) at (0,1) {$\mu_{1\ldots k-1}$};
  \node[grey] (nu) at (0,-1) {$\nu_{1\ldots k-1}$};
  \node[grey] (alpha) at (3,1) {$\alpha$};
  \node[grey] (beta) at (3,-1) {$\beta$};
  \node[grey] (zeta) at (6,0) {$\zeta_i$};
  \node[grey] (delta) at (6,2) {$\delta_{ij}$};
  \node[grey] (epsilon) at (6,-2) {$\epsilon_{im}$};

  \draw[to] (mu) .. controls (1,0) and (2, 2) .. node[near end, above] {$[\bar \mu, \mu_{k\ldots\vert \mu \vert}]$} (alpha); 
  \draw[to] (nu) .. controls (1,0) and (2, -2) .. node[near end, below] {$[\bar \nu, \nu_{k\ldots\vert \nu \vert}]$} (beta); 

  \draw[to] (alpha) .. controls (4,0) and (5, 1) .. node[near end, above] {$[\bar  \gamma_i^1, \gamma_i]$} (zeta); 
  \draw[to] (beta) .. controls (4,0) and (5, -1) .. node[near end, below] {$[\bar  \gamma_i^2, \gamma_i]$} (zeta); 

  \draw[to,dashed] (alpha) .. controls (4,1) and (5, 3) .. node[near end, above] {$[\bar  \gamma_i^1, \eta_{ij}]$} (delta); 
  \draw[to,dashed] (beta) .. controls (4,-1) and (5, -3) .. node[near end, below] {$[\bar  \gamma_i^2, \xi_{im}]$} (epsilon); 
  
\end{tikzpicture}
\end{center}
\caption{Construction of the paths considered in the proof of Theorem \ref{thm_automorphism} associated to a given $\gamma_i$.}
\label{fig_SmuSnu}
\end{figure}
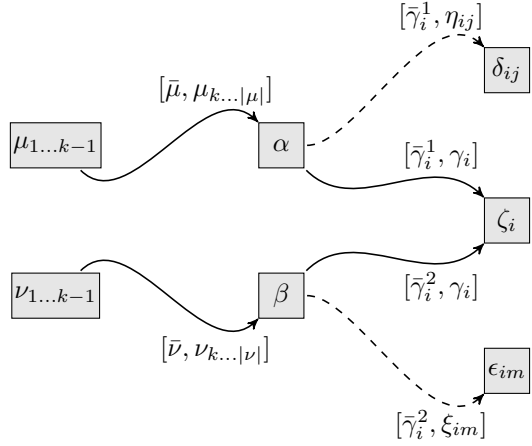

By Lemma \ref{lem_Etau_properties}, for each $\gamma \in E^\ast$ with
$\source_E(\gamma) = \range_E(\mu)$, there is precisely one path in $E^*_\tau$ with
second label $\gamma$ and source $\alpha$. Similarly, there is
precisely one path in $E^*_\tau$ with second label $\gamma$ and source
$\beta$.
Let $\{ \gamma_i \}$ be the paths in $E^l$ with $\source_E(\gamma) = \range_E(\mu) =
\range_E(\nu)$. Given such a $\gamma_i$, let $\bar \gamma_i^1$ be the first
label of the unique path in $E^*_\tau$ with source $\alpha$ and second
label $\gamma_i$. Similarly, let $\bar \gamma_i^2$ be the first label
of the unique path in $E^*_\tau$ with source $\beta$ and second label
$\gamma_i$. Right-synchronization of $(E_\tau,\LL_2)$ implies that
these two paths have the same range $\zeta_i$. This is shown in Figure \ref{fig_SmuSnu} for a single path $\gamma_i$.

Left-synchronization of $(E_\tau,\LL_1)$ implies that any path in
$E^\ast_\tau$ with first label $\bar \gamma_i^1$ must start at $\alpha$
and right-resolvancy of $(E_\tau,\LL_1)$ means that all these
paths have different second labels. Let $\{ \gamma_i \} \cup \{
\eta_{ij} \}$ be the set of second labels of these paths, and let
$\delta_{ij}$ be the range of the path with second label $\eta_{ij}$. This is illustrated by a dashed path in Figure \ref{fig_SmuSnu} for a single value of $j$.
Similarly, any path in $E^*_\tau$ with first label $\bar \gamma_i^2$
must start at $\beta$ and all these paths have unique second
labels. Let $\{ \gamma_i \} \cup \{\xi_{im}\}$ be the the set of
second labels of these paths, and let $\epsilon_{im}$ be the range of
the path with second label $\xi_{im}$. As above, this is illustrated by a dashed path in Figure \ref{fig_SmuSnu} for a single value of $m$.

For each $\gamma_i$ with $\source_E(\gamma_i) = \range_E(\mu) = \range_E(\nu)$ and $\vert
\gamma_i \vert = l$, use Lemma \ref{lem_Pmu} to choose $q_{\mu 
  \gamma_i}$ and $q_{\nu \gamma_i}$ such that $\lambda_\tau(q_{\mu
  \gamma_i}) = P_{\mu \gamma_i}$ and $\lambda_\tau(q_{\nu \gamma_i}) =
P_{\nu \gamma_i}$. By Proposition \ref{prop_action_on_path},
\begin{multline*}
\lambda_\tau \left( \sum_i q_{\mu \gamma_i} 
S_{\bar \mu} S_{\bar \gamma_i^1} S_{\bar \gamma_i^2}^\ast  S_{\bar \nu}^\ast 
q_{\nu \gamma_i} \right) \\
= \sum_i P_{\mu \gamma_i}
S_\mu
\left( S_{\gamma_i} S_{\zeta_i}^\ast + \sum_j S_{\eta_{ij}} S_{\delta_{ij}}^\ast   \right)
\left( S_{\zeta_i} S_{\gamma_i}^\ast + \sum_m S_{\epsilon_{im}} S_{\xi_{im}}^\ast   \right)
S_\nu^\ast
P_{\nu \gamma_i} \\
= \sum_i S_{\mu \gamma_i} S_{\nu \gamma_i}^\ast = S_\mu S_\nu^\ast.
\end{multline*}
Considering a sum of such elements allows the construction of a permutative unitary $U_\pi$ such that $\lambda_\tau(U_\pi) = U_\tau^\ast$.  Hence, $\lambda_\tau$ is invertible with inverse $\lambda_\pi$.

$\enumref{thm_auto_auto} \implies \enumref{thm_auto_synch}$: This was proved in
\cite[Thm.~6.4]{chs_endomorphisms_of_graph_algebras} up to a slight
reformulation similar to the one given in Remark \ref{rem_chs_graph}
which is needed to translate the result into a statement about
permutation graphs.
\end{proof}

\noindent Note how this result shows that for a permutative endomorphism $\lambda_\tau$ of $C^\ast(E)$, invertibility of the induced endomorphism of the
shift space $\X_E$ can be lifted and used to constructively find an inverse to $\lambda_\tau$. 

\begin{example}\label{ex_cycles}
  Consider again the graph $E$ and the permutation $\tau$ from Example
  \ref{ex_golden_mean}. It is straightforward to check that for any path $A$ in $E_\tau^3$, the label $\LL_1(A)$ uniquely determines the source $\source_\tau(A)$,
  so the permutation graph is synchronizing in the first label. On the
  other hand, $(E_\tau, \LL_2)$ contains a cycle labeled $11$ and a
  loop labeled $1$, so by Lemma \ref{lem_synch}, the permutation graph
  is not synchronizing in the second label. Hence, $\lambda_\tau$ is not an automorphism. This fact was also observed in
  \cite{chs_endomorphisms_of_graph_algebras} without reference to the
  permutation graph.
\end{example}

\begin{remark}
\label{rem_sync}
As in Example \ref{ex_cycles}, Lemmas \ref{lem_synch}, \ref{lem_Pmu} and Theorem \ref{thm_automorphism} are
useful in general for automorphism testing.
 A forthcoming paper will present an efficient algorithm that 
 decides whether a permutative endomorphism $\lambda_\tau$ is an automorphism, 
 by way of its permutation graph.
 This is used in the computer program discussed in Section \ref{sec_bowtie}
 together with the algorithm from Section \ref{sec_Otau_alg} to perform
 exhaustive searches for permutative automorphisms at higher levels than 
 otherwise possible.
\end{remark}

\subsection{Induced endomorphism of the one- and two-sided shift.}\label{sec_induced}

When a permutative endomorphism restricts to an automorphism of $\DD_E$, 
i.e.\ when its permutation graph is synchronizing in the first label, 
it induces an endomorphism $\phi_\tau^+$ on the one-sided edge shift $\X_E^+$
through the identification of $z^+ \in \X^+_E$ with the map $\phi_{z^+} \colon
\DD_E \to \C$. This in turn induces an endomorphism $\phi_\tau$ on the
two-sided edge shift $\X_E$.

\begin{prop}\label{prop_one-sided}
  Let $\lambda_\tau$ be a permutative endomorphism at level $k$ for
  which the permutation graph is synchronizing in the first label, and
  let $\phi_\tau^+$ be the induced endomorphism of the one-sided edge
  shift $\X^+_E$. 
  Given $x^+ \in \X^+_E$, let $\beta \in E^{k-1}$ be the
  source of the unique infinite path in $E^*_\tau$ with
  first label $x^+$, and let $y^+$ be the second label of this path. Then 
  $\phi_\tau^+(x^+) = \beta y^+$.
  \end{prop}

\begin{proof}
  Synchronization in the first label guarantees that there is a unique
  right-infinite path in the permutation graph with first label
  $x^+$. The result follows from Proposition \ref{prop_action_on_path}
  and the identification of $z^+ \in \X^+_E$ with the map $\phi_{z^+}
  \colon \DD_E \to \C$.
\end{proof}

\noindent As described in Section \ref{sec_textile}, the textile system $T_\tau$ corresponding to a permutation graph codes an endomorphism $\phi_{T_\tau}$ of $\X_E$. As an immediate consequence of the previous result, this endomorphism can be linked to the natural endomorphism of $\X_E$ induced by $\lambda_\tau$:

\begin{cor}\label{cor_two-sided}
  Let $\lambda_\tau$ be a permutative endomorphism at level $k$ for
  which the permutation graph is synchronizing in the first label, let
  $\phi_\tau$ be the induced endomorphism of the two-sided edge shift
  $\X_E$, and let $\phi_{T_\tau}$ be the
  endomorphism coded by $T_\tau$. Then $\phi_\tau(x) = \sigma^{-k+1}
  \circ \phi_{T_\tau}$.
\end{cor}

\subsection{Composition of permutation graphs.}\label{sec_composition}
Let $\lambda_\tau$ and $\lambda_\pi$ be permutative endomorphisms at
level $l$ and $k$, respectively. Given $e \in E^1$, consider all
$\alpha_i,\beta_i \in E^{k-1}$ and $f_i \in E^1$ for which $E_\tau$
contains an edge $\beta_i \llarc{e}{f_i} \alpha_i$.
%
%
%
By Lemma \ref{lem_image_Se}, $\lambda_\pi(S_e) = \sum_i S_{\beta_i} S_{f_i} S_{\alpha_i}$.

For each $i$, let $j$ enumerate the pairs of paths $A_{ij}, B_{ij}$ over $(E_\tau,\LL_\tau)$ for which
\begin{itemize}
\item the first label of $A_{ij}$ is $\alpha_i$,
\item the first label of $B_{ij}$ is $\beta_i f_i$, and
\item $A_{ij}$ and $B_{ij}$ have the same range.
\end{itemize}
For each $j$, 
let $\delta_{ij}$ be the source of $A_{ij}$,
let $\gamma_{ij}$ be the source of $B_{ij}$,
let $\alpha_{ij}'$ be the second label of $A_{ij}$, and 
let $\beta_{ij}' g_{ij}$ be the second label of $B_{ij}$.
%
Such a pair of paths is illustrated here:
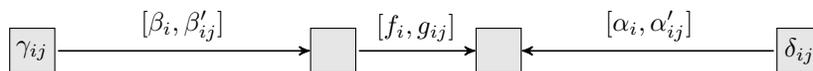
\begin{center}
  \begin{tikzpicture}
  [bend angle=10,
  clearRound/.style = {circle, inner sep = 0pt, minimum size = 17mm},
  clear/.style = {rectangle, minimum width = 5 mm, minimum height = 5 mm, inner sep = 0pt},  
  greyRound/.style = {circle, draw, minimum size = 1 mm, inner sep =
    0pt, fill=black!10},
  grey/.style = {rectangle, draw, minimum size = 6 mm, inner sep =
    1pt, fill=black!10},
  white/.style = {rectangle, draw, minimum size = 6 mm, inner sep =
    1pt},
  to/.style = {->, shorten <= 1 pt, >=stealth', semithick}]  
  
  \node[grey] (gamma) at (0,-1.5) {$\gamma_{ij}$};  
  \node[grey] (M1) at (4,-1.5) {}; 
  \node[grey] (M2) at (6.2,-1.5) {};  
  \node[grey] (delta) at (10.2,-1.5) {$\delta_{ij}$};  

  \draw[to] (gamma) to node[auto] {$[\beta_i,\beta_{ij}']$} (M1);
  \draw[to] (M1) to node[auto] {$[f_i,g_{ij}]$} (M2);
  \draw[to] (delta) to node[auto,swap] {$[\alpha_i,\alpha_{ij}']$} (M2);
  
\end{tikzpicture}
\end{center}
By Proposition \ref{prop_action_on_path}, 
\begin{displaymath}
\lambda_\tau(\lambda_\pi(S_e))
  = \sum_{i,j} S_{\gamma_{ij}} S_{\beta'_{ij}} 
                     S_{g_{ij}} 
                     S_{\alpha'_{ij}}^\ast S_{\delta_{ij}}^\ast,
\end{displaymath}
so the permutation graph of $\lambda_\tau \circ \lambda_\pi$ will have vertices and edges as illustrated here:
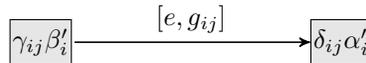
\begin{center}
\begin{tikzpicture}
  [bend angle=10,
  clearRound/.style = {circle, inner sep = 0pt, minimum size = 17mm},
  clear/.style = {rectangle, minimum width = 5 mm, minimum height = 5 mm, inner sep = 0pt},  
  greyRound/.style = {circle, draw, minimum size = 1 mm, inner sep =
    0pt, fill=black!10},
  grey/.style = {rectangle, draw, minimum size = 6 mm, inner sep =
    1pt, fill=black!10},
  white/.style = {rectangle, draw, minimum size = 6 mm, inner sep =
    1pt},
  to/.style = {->, shorten <= 1 pt, >=stealth', semithick}]  
  
  \node[grey] (gammabeta) at (0,-3) {$\gamma_{ij} \beta'_i$};  
  \node[grey] (deltaalpha) at (4,-3) {$\delta_{ij} \alpha'_i$};  
  \draw[to] (gammabeta) to node[auto] {$[e,g_{ij}]$} (deltaalpha);  
  
\end{tikzpicture}
\end{center}
Note that this has the form of a permutation graph for a permutation at
level $k+l-1$. The corresponding permutation is determined by this
permutation graph, but it is not easily computed directly from $\tau$
and $\pi$ without going through this process. The following result
summarizes this discussion:

\begin{prop}
  Let $\lambda_\tau$ and $\lambda_\pi$ be permutative
  endomorphisms. Then the permutation graph of $\lambda_\tau \circ
  \lambda_\pi$ is the labeled graph constructed above.
\end{prop}

\section{Finding permutative automorphisms}
\label{sec_Etau_alg}

This section describes a recursive algorithm that, given a graph $E$
and a level $k$, finds all level-$k$ permutative automorphisms of $C^\ast(E)$.  Only
a high-level description is given here; a detailed description of the
algorithm will be published separately. 
The overall structure follows two simple mutually recursive functions
$\completeEtau$ and $\completeletter$, defined in Equation \eqref{eq:compute-Etau}, that
build all valid extensions to a partial graph $\GG$ edge by edge.
It will be shown below that $\completeEtau\left(0,\GG_0\right)$ precisely constructs 
the set of permutation graphs of level $k$ when $\GG_0$ is the empty labeled graph with vertex set $E^{k-1}$.
For each $r\in E^0$, write $E_{r\to\ast}^{k-1} = \setof{\alpha^n_r}{ 0\le n \le |E^{k-1}_{r\to\ast}|-1}$
and $E^1 = \setof{e_m}{0\le m \le |E^1|-1}$. Let $r_m = \range_E(e_m)$ for
each $m$, and let $\GG = (G,\LL)$ denote a partially 
completed permutation graph. 
\begin{equation}
  \label{eq:compute-Etau}
  \begin{split}
    \completeEtau(\vert E^1 \vert,\GG) &= \{ \GG \}\\
    \completeEtau(m,\GG) &= \completeletter(m,0,\GG)\\
    \\
    \completeletter(m,|E^{k-1}_{r\to\ast}|,\GG) &= \completeEtau(m+1,\GG)\\
    \completeletter(m,n,\GG) &= \disjointunion_{
      \begin{smallmatrix}
      {\mu\colon \beta\,\llarc{e_m}{f}\,\alpha_{r_m}^n}\\
      {\isvalid{\addedge{\GG}{\mu}}}
    \end{smallmatrix}
  }
    \completeletter(m,n+1,\addedge{\GG}{\mu}) 
  \end{split}
\end{equation} 
In the above, ``$\uplus$'' denotes disjoint set union, and $\GG\oplus\mu$ is the labeled
graph resulting from extending $\GG$ with the labeled edge $\mu$. Each edge $\mu$ is added only if $\GG\oplus\mu$ satisfies a predicate $\isvalid{\cdot}$ defined below, 
which guarantees that the construction only results in permutation graphs for automorphisms. 
%
%
By Lemma \ref{lem_Etau_properties}, each vertex $\alpha\in
E^{k-1}_{r\to\ast}$ must receive exactly one edge with first label $e$
for each $e\colon s\to r$ in $E^1$. In the algorithm defined above, the graphs are constructed one
label $e_m\colon s_m\to r_m$ at a time, placing for each destination
vertex $\alpha^n_{r_m}$ the $e_m$-labeled edge incident to it: For
each $\beta$ and $f$ for which adding the edge
$\mu\colon\beta\llarc{e_m}{f}\alpha^n_{r_m}$ results in a valid
subgraph of a permutation graph, the recursion proceeds.  Dead ends result in the empty set,
but if all labels are successfully completed, the singleton set
containing the completed graph is returned. At each recursion level,
the result is the disjoint union of the results from the levels below.

The predicate $\isvalid{\cdot}$ ensures that it is exactly the
permutative automorphisms that are constructed. It guarantees that only
edges that satisfy two concurrent conditions are placed: The resulting
graph must be a subgraph of a permutation graph and, by Theorem
\ref{thm_automorphism}, it must also synchronize in both labels in
order to complete to an automorphism. 
The following proposition, which is an easy corollary to Proposition
\ref{prop_sufficient}, characterizes the subgraphs of permutation graphs:

\begin{prop}\label{prop_Etau_subgraph}
  Let $(G,\LL)$ be a labeled graph with node set $E^{k-1}$ and label alphabet 
  $E^1\times E^1$. $(G,\LL)$ is subgraph of a permutation graph $(E_\tau,\LL_\tau)$ for some endpoint-fixing
  permutation $\tau\in\Perm(E^k)$ if and only if
  \begin{enumerate}
  \item \label{prop_Etau_subgraph_i}
    $(G,\LL)$ satisfies parts (\ref{lem_Etau_properties_L1}) and (\ref{lem_Etau_properties_L2}) of  Lemma \ref{lem_Etau_properties}.
  \item \label{prop_Etau_subgraph_ii}
    For every edge $e\colon s\to r$ in $E^1$ and $\alpha\in E^{k-1}_{r\to\ast}$, there is at most one
    edge $\mu\colon \beta \llarc{e}{f} \alpha$ in $(G,\LL)$ with $f\in E^1$ and $\beta\in E^{k-1}$.
  \item \label{prop_Etau_subgraph_iii}
    For every edge $f\colon q\to t$ in $E^1$ and $\beta\in E^{k-1}_{\ast\to q}$, there is at most one 
    edge $\mu\colon \beta \llarc{e}{f} \alpha$ in $(G,\LL)$ with $e\in E^1$ and $\alpha\in E^{k-1}$.
  \end{enumerate}
\end{prop}

This yields a simple test for whether a partially completed graph is
the subgraph of some permutation graph that corresponds to an automorphism:
\begin{definition}
  \label{def_isvalid}
  Let $\GG = (G,\LL)$ be a labeled graph with node set $E^{k-1}$ and label alphabet $E^1\times E^1$.
  $\isvalid{\GG}$ is true if $\GG$ synchronizes in both
  labels and satisfies the conditions
  of Proposition \ref{prop_Etau_subgraph}.
\end{definition}

The correctness of the recursive construction of all permutative automorphisms
depends on the following property of the $\isvalid{\cdot}$ predicate:
\begin{definition} 
  \label{def_local_property}
  A predicate $\predicate{p}$ on the set of labeled graphs is said to
  be \emph{inherited} if $\predicate{p}(H,\LL) \implies 
  \predicate{p}(G,\LL) $ for every labeled graph $(H,\LL)$ and 
  subgraph $G$ of $H$.
\end{definition}
\noindent%
That is, the property is invariant throughout an edge-by-edge
construction of the labeled graph. As soon as a labeled edge is placed
that makes the partially constructed graph invalid, the entire search
tree below it may be safely discarded as it can be contained in no
valid completed graph. Conversely, any valid completed graph is reached
from the empty graph (or any other subgraph) by a sequence of locally
valid edge placements, the order of which is unimportant.

It is easy to verify that Properties
\enumref{prop_Etau_subgraph_i}--\enumref{prop_Etau_subgraph_iii} of
Proposition \ref{prop_Etau_subgraph} are inherited properties. Left- and
right-synchronization are inherited properties by way of Lemma
\ref{lem_synch}, since the set of cycles over $H$ is a subset of the
cycles over $G$ when $H$ is a subgraph of $G$.

\begin{thm}
  \label{thm_Pgraph_construct}
  Let  $\GG_0$ be the labeled graph with vertex set $E^{k-1}$ and no edges. Then
  $\completeEtau\left(0,\GG_0\right)$ is the set of permutation
  graphs corresponding to all level $k$ permutative automorphisms.
\end{thm}

\begin{proof}
  Assume that $\GG\in\completeEtau\left(0,\GG_0\right)$. 
  It is immediately apparent from Equation \eqref{eq:compute-Etau} that the total number of edges in $\GG = (G,\LL)$ must be
  \[
  \left|G^1\right| 
  = 
    \sum_{
      \begin{smallmatrix}
      e\colon s\to r\\
      e \in E^1 
      \end{smallmatrix} 
    } \left|E^{k-1}_{r\to\ast}\right| 
  = 
    \sum_{
      \begin{smallmatrix}
      f\colon q\to t\\
      f \in E^1 
      \end{smallmatrix} 
    } \left|E^{k-1}_{\ast\to q}\right|. 
  \]
  Proposition
  \ref{prop_Etau_subgraph}\enumref{prop_Etau_subgraph_ii} ensures that
  for each $e\colon s\to r$, every $\alpha\in E^{k-1}_{r\to\ast}$ receives at most one edge with
  first label $e$, and the full edge count can only be achieved if
  each $\alpha$ receives exactly one such edge, fulfilling Lemma
  \ref{lem_Etau_properties}\enumref{lem_Etau_properties_receive}. Similarly,
  Proposition \ref{prop_Etau_subgraph}\enumref{prop_Etau_subgraph_iii},
  together with the second equality above, ensures that Lemma
  \ref{lem_Etau_properties}\enumref{lem_Etau_properties_emit} is
  satisfied.
  Consequently, $\GG$ satisfies all properties
  of Lemma \ref{lem_Etau_properties}, and by Proposition
  \ref{prop_sufficient}, it is a permutation graph. Since it also
  synchronizes in both labels, it represents a permutative
  automorphism by Theorem \ref{thm_automorphism}.

  Conversely, let $(E_\tau,\LL_\tau)$ be a permutative
  automorphism at level $k$.  By Proposition \ref{prop_Etau_subgraph} and the fact that $\isvalid{\cdot}$ is inherited, the edge set $E^1_\tau$ traversed in any order
  constitutes a sequence of locally valid edge placements that
  incrementally extends $\GG_0$ until $(E_\tau,\LL_\tau)$ is
  completed.  Hence, each time the last line in Equation
  \eqref{eq:compute-Etau} is reached, there is a valid edge placement to
  pick from $E^1_\tau$.  Since the procedure terminates after 
  $\sum_{e\in E} \left|E^{k-1}_{\range_E(e)\to\ast}\right| = \left|E^1_\tau\right|$
  steps, all the edges of $E^1_\tau$ are eventually placed, and thus
  $(E_\tau,\LL_\tau)\in\completeEtau\left(0,\GG_0\right)$.

  Consequently, the procedure constructs exactly the set of all level-$k$ permutative automorphisms.    
\end{proof}
From the recursion structure in Equation \eqref{eq:compute-Etau}, it is apparent that each 
permutation graph is built at most once. Because all level-$k$ permutative automorphisms are reached,
each is constructed exactly once.

It is instructive to note that the permutation graph property (Lemma
\ref{lem_Etau_properties}) and the automorphism property (Theorem
\ref{thm_automorphism}) are separate tests: Simply omitting the synchronization
test in the above procedure would instead yield the set of all
permutative endomorphisms at level $k$. While this set becomes too
large to practically compute even for small $k$, other properties than
automorphism may be filtered for, just so long as they can
be expressed in an inherited form as defined in Definition
\ref{def_local_property}.

\section{Inner equivalence, shift space equivalence, and order}
\label{sec_inner_order}
One is most often not interested in every automorphism, but only in
those that are sufficiently different to warrant distinct
consideration. Commonly, endomorphisms are considered up to inner
equivalence, i.e.\ modulo conjugation by a unitary. The present
section introduces \emph{shift space equivalence}, which groups
endomorphisms that have identical properties as dynamical systems:

\begin{definition}[Shift space equivalence]\label{def:shiftspace-eq}
  The labeled graphs $(E,\LL_E)$ and $(F,\LL_F)$ are said to be shift space equivalent if they present the same shift space $\X_{(E,\LL_E)} = \X_{(F,\LL_E)}$.
\end{definition}

Definition \ref{def:shiftspace-eq} induces an equivalence
relation on $\Perm(E^k)$ and on the corresponding permutative
endomorphisms: Two level-$k$ endpoint-fixing permutations $\tau$ and $\tau'$
are said to be shift space equivalent when their respective permutation graphs
are equivalent, i.e.\ when $\X_{(E_\tau,\LL_\tau)} = \X_{(E_{\tau'},\LL_{\tau'})}$.
For permutative endomorphisms that are synchronizing in the first label, it will be shown that 
shift space equivalence is the same as inner equivalence via a permutative unitary.


\subsection{Permuting permutations.}
Let $\tau\in\Perm(E^k)$ be endpoint-fixing. Given an endpoint-fixing
permutation $\pi\in\Perm(E^{k-1})$, the aim of the following is to
rigorously construct a new permutation $g_\pi(\tau)$, the permutation
graph of which is the graph obtained from
$(E_\tau,\LL_\tau)$ by permuting the vertices through $\pi$ while
maintaining the remaining structure. Later, it will be shown that when
$\pi$ ranges over the endpoint-fixing permutations in
$\Perm(E^{k-1})$, this construction exactly reaches all of the
level-$k$ permutations that are shift space equivalent to $\tau$ whenever $(E_\tau,\LL_\tau)$ is synchronizing in the first label.

Let $\tau \in \Perm(E^{k})$ and $\pi \in \Perm(E^{k-1})$ be endpoint-fixing. Given $e \in E^1$ and $\alpha \in E^{k-1}$ with $e \alpha \in E^k$, consider the unique $\beta \in E^{k-1}$ and $f \in E^1$ for which $\tau(e\alpha) = \beta f$. Define $g_\pi(\tau) \colon E^k \to E^k$ by 
\begin{equation}\label{eq_f_pi}
g_\pi(\tau)(e \pi(\alpha)) = \pi(\beta)f.
\end{equation}
%
The paths used in this definition are sketched in Figure \ref{fig_fpi}.

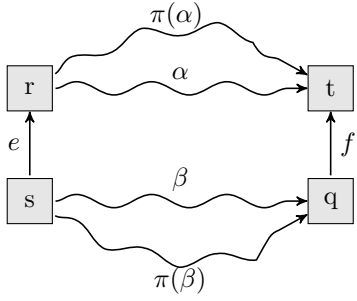
\begin{figure}
\begin{center}
\begin{tikzpicture}
  [bend angle=10,
   clearRound/.style = {circle, inner sep = 0pt, minimum size = 17mm},
   clear/.style = {rectangle, minimum width = 5 mm, minimum height = 5 mm, inner sep = 0pt},  
   greyRound/.style = {circle, draw, minimum size = 1 mm, inner sep =
      0pt, fill=black!10},
   grey/.style = {rectangle, draw, minimum size = 6 mm, inner sep =
      1pt, fill=black!10},
    white/.style = {rectangle, draw, minimum size = 6 mm, inner sep =
      1pt},
   to/.style = {->, shorten <= 1 pt, >=stealth', semithick}]
  
  \node[grey] (se) at (0,0) {s};
  \node[grey] (re) at (0,1.5) {r};
  \node[grey] (stau) at (4,0) {q};
  \node[grey] (rtau) at (4,1.5) {t};

  \draw[to] (se) to node[auto] {$e$} (re);

  \draw[to,snake it] (re) to node[auto] {$\alpha$} (rtau);
  \draw[to,snake it,bend left=30] (re) to node[auto,anchor=south] {$\pi(\alpha)$\;} (rtau); 

  \draw[to,snake it] (se) to node[auto] {$\beta$} (stau);
  
    \draw[to,snake it, bend right=30] (se) to node[auto,swap] (pibeta) {
    } (stau); 
    
    \node[clear,anchor=north] at (pibeta) {$\pi (\beta)$};

    \draw[to] (stau) to node[auto,swap] {$f$} (rtau);  
\end{tikzpicture}
\end{center}
\caption{Paths and edges in $E$ used in the construction of $g_\pi(\tau)$ and in the proof of Lemma \ref{lem_fpi}. Compare to Figure \ref{fig_Etau}. }
\label{fig_fpi}
\end{figure}

\begin{lem}\label{lem_fpi}
For each endpoint-fixing $\tau \in \Perm(E^{k})$ and $\pi\in\Perm(E^{k-1})$, $g_\pi(\tau)$ is also an endpoint-fixing permutation of  $E^{k}$, and the edges of $(E_{g_\pi(\tau)},\LL_{g_\pi(\tau)})$ are $\mu_\pi \colon \pi(\beta)\llarc{e}{f}\pi(\alpha)$ for each edge
$\mu\colon \beta\llarc{e}{f}\alpha$ of $(E_\tau,\LL_\tau)$.
\end{lem}

\begin{proof}
Consider Figures \ref{fig_Etau} and \ref{fig_fpi}, and notice that $g_\pi(\tau)$ is an endpoint-fixing permutation since both $\pi$ and $\tau$ are endpoint-fixing. 
\end{proof}

\noindent In this way, $(E_\tau,\LL_\tau)$ and $(E_{g_\pi(\tau)}, \LL_{g_\pi(\tau)})$ are constructed to have the same structure, and this immediately yields the following:

\begin{cor}\label{lem_orderandsynch}
For endpoint-fixing $\tau \in \Perm(E^k)$ and $\pi \in \Perm(E^{k-1})$, $(E_\tau,\LL_\tau)$ is synchronizing in the first/second label if and only if $(E_{g_\pi(\tau)},\LL_{g_\pi(\tau)})$ is synchronizing in the first/second label.
\end{cor}

The following proposition shows that the construction given above
captures shift space equivalence of permutative endomorphisms which are synchronizing in the first label; and that shift space equivalence of such endomorphisms is the same as inner equivalence
through a permutative unitary $U_\pi$ with $\pi\in\Perm(E^{k-1})$.
This relation to inner equivalence has already been considered elsewhere. In particular, \cite[Sec. 4.2]{conti_szymanski_labeled_trees} contains a discussion of the action of such inner automorphisms in the context of Cuntz algebras.

\begin{prop}\label{prop_SSE}
Let $\tau, \tau' \in \Perm(E^k)$ be  endpoint-fixing, and assume that the corresponding permutation graphs are synchronizing in the first label. The following are equivalent:
\begin{enumerate}
\item $\tau$ and $\tau'$ are shift space equivalent, i.e.\ $\X_{(E_\tau, \LL_\tau)} = \X_{(E_{\tau'}, \LL_{\tau'})}$.\label{vlabels_shift}
\item There exists $\pi\in\Perm(E^{k-1})$ such that $\tau' = g_\pi(\tau)$. \label{vlabels_fpi}
\item There exists an endpoint-fixing $\pi \in  \Perm(E^{k-1})$ such that $\Ad(U_\pi) \circ \lambda_\tau = \lambda_{\tau'}$.\label{vlabels_Upi}
\item The shift space automorphisms coded by the textile systems are equal, i.e.\ $\phi_{T_\tau} = \phi_{T_{\tau'}}$. \label{vlabels_phiTtau}
\item The shift space automorphisms induced by $\lambda_\tau$ and $\lambda_\tau'$ are equal, i.e.\ $\phi_{\tau} = \phi_{\tau'}$. \label{vlabels_phitau}
\end{enumerate}  
\end{prop}

\begin{proof}\ 

\enumref{vlabels_fpi} $\Rightarrow$ \enumref{vlabels_shift}: 
This follows from Lemma \ref{lem_fpi}.

\enumref{vlabels_shift} $\Rightarrow$ \enumref{vlabels_fpi}: 
By Lemma \ref{lem_labeltosr}, there exists an $n \in \N$ such that 
for a path of length $n$ in $(E_\tau,\LL_\tau)$ or
$(E_{\tau'},\LL_{\tau'})$, its label uniquely determines the range and source of
the path. Hence, a vertex in $E_\tau^0$ or $E_{\tau'}^0$ can be
identified with the set of labels for length $n$ paths that it
emits.  Since the languages presented by $(E_\tau,\LL_\tau)$ and 
$(E_{\tau'},\LL_{\tau'})$ are equal by assumption, this induces a
bijection $\pi$ from $E^{k-1} = E_\tau^0$ to $E^{k-1} =
E_{\tau'}^0$. By construction, $\pi$ is endpoint-fixing and $\tau' =
g_\pi(\tau)$.

\enumref{vlabels_fpi} $\Leftrightarrow$ \enumref{vlabels_Upi}: 
Assume that $g_\pi(\tau) = \tau'$. Let $e \in E^1$ be given and use
Lemma \ref{lem_image_Se} to obtain $\alpha_i, \beta_i \in E^{k-1}$ and
$f_i \in E^1$ such that $\lambda_\tau(S_e) = \sum_i S_{\beta_i}
S_{f_i} S_{\alpha_i}$. Then,
\[
  \begin{split} 
   (\Ad(U_\pi) \circ \lambda_\tau) (S_e) 
       &= \big( \sum_{\alpha} S_{\pi(\alpha)} S_\alpha^\ast \big)
          \big( \sum_i S_{\beta_i} S_{f_i} S_{\alpha_i}^\ast \big)
          \big( \sum_{\alpha} S_\alpha S_{\pi(\alpha)}^\ast \big) \\
      &=  \sum_i S_{\pi(\beta_i)} S_{f_i} S_{\pi(\alpha_i)}^\ast    \\
      &= \lambda_{\tau'}(S_e). 
  \end{split}
\]
The reverse implication follows by an analogous argument.

\enumref{vlabels_shift} $\Leftrightarrow$ \enumref{vlabels_phiTtau}: 
Let $x \in \X_E$. By definition, $\phi_{T_\tau}(x)$ is the unique $y \in \X_E$ for which $(x_i,y_i)_{i \in \Z} \in \X_{(E_\tau,\LL_\tau)}$. Hence, $\phi_{T_\tau} = \phi_{T_{\tau'}}$ if and only if $\X_{(E_\tau, \LL_\tau)} = \X_{(E_{\tau'}, \LL_{\tau'})}$. Finally, \enumref{vlabels_phiTtau} and \enumref{vlabels_phitau}  are equivalent by Corollary \ref{cor_two-sided}.
\end{proof}

\subsection{Vertex order and shift space equivalence.}
Proposition \ref{prop_SSE}\enumref{vlabels_fpi} shows that the
shift space equivalence relation is a refinement of the inner equivalence relation,
corresponding to inner equivalence through permutative unitaries. 
Thus, every class at level $k$ has the same size, namely the
number of endpoint-fixing permutations at level $(k-1)$. This number
grows combinatorially in $k$, which prompts two questions: How
can one recognize when endpoint-fixing permutations
$\tau,\tau'\in\Perm(E^k)$ belong to the same shift space equivalence
class without testing every endpoint-fixing $\pi\in\Perm(E^{k-1})$?
And is it possible to define a succinct representation of the
equivalence classes that can be constructed directly, 
without considering their many individual elements?


Assume that there exists a strict total ordering -- i.e.\ an
enumeration -- of the vertices in a permutation graph, for which the
position in the order of any given vertex depends only on the labeled
paths emitted and/or received by this vertex.  Such an order is
determined by the language of the permutation graph. Assuming that
$\X_{(E_\tau,\LL_\tau)} = \X_{(E_{\tau'},\LL_{\tau'})}$, this allows
direct construction of $\pi$ from Proposition \ref{prop_SSE} by pairing the vertices that have the
same number in the order. 
In fact, replacing each vertex
in a permutation graph with its number in the order will be shown to
yield a canonical representation for an entire shift space equivalence
class. The resulting structure is called an \emph{ordered permutation graph},
and a precise definition will be given below.

From a computational viewpoint, it is greatly desirable if the ordered
permutation graphs can be constructed directly, such that the equivalence
classes can be studied without first considering the many individual
members, and this motivates the definition of the concrete ordering
introduced below. The aim is to order the vertices of a permutation
graph in a way that facilitates recursive construction of collections
of shift space equivalence classes using an algorithm structurally similar
to the one introduced in Section \ref{sec_Etau_alg}. Since different
equivalence classes have different languages and hence different
orders, this requires that the order can be built incrementally 
alongside the ordered permutation graph. The concrete order
given below makes it easy to construct ordered permutation
graphs by adding each edge \emph{in order}, building at the same
time the ordered permutation graph and its corresponding total order. An algorithm carrying out this  construction will be detailed in Section \ref{sec_Otau_alg}.

\subsubsection{The order of minimal emitted sequences.}

Begin by ordering the vertices of $E$, and introduce a strict total
order on the edges of $E$ for which $e \leq f$ for $e,f \in E^1$
if $\source_E(e) < \source_E(f)$ or $\source_E(e) = \source_E(f)$ and $\range_E(e) \leq \range_E(f)$.
These properties uniquely determine the order if and only if $E$ has
no parallel edges, but in general, the edge order involves an
arbitrary choice. This order on vertices and edges of $E$ is then
used to construct a special family of orderings of the vertices and 
edges of the permutation graphs $E_\tau$ by the following construction.

\begin{definition}[Minimal emitted sequence]
\label{def_order}
Let $\tau\in\Perm(E^k)$ be endpoint-fixing. Given $\alpha \in E_\tau^0$, define $\MESv{\tau}(\alpha) =
e_1\range_E(f_1)e_2\range_E(f_2)\cdots$ to be the lexicographically minimal sequence for which there exists
an infinite path  
\[
 A \colon \alpha \llarc{e_1}{f_1} \alpha_1 \llarc{e_2}{f_2} \alpha_2\llarc{e_3}{f_3} \cdots
\]
over $(E_\tau,\LL_\tau)$.
Such a path is called a \emph{minimal emitted path}, and the sequence $\MESv{\tau}(\alpha)$ is called the \emph{minimal emitted sequence} of $\alpha$.
Lexicographic order of the minimal emitted sequences defines a total preorder on $E_\tau^0$,
the \emph{preorder of minimal emitted sequences}, given by $\alpha \leq \beta$ if and only if $\MESv{\tau}(\alpha) \leq \MESv{\tau}(\beta)$.
\end{definition}

The preorder of minimal emitted sequences is lifted to permutation graph edges in the obvious way:
\begin{definition}\label{def_edge_order}
Given a labeled edge $\mu\colon \beta\llarc{e}{f}\alpha$ of $(E_\tau,\LL_\tau)$, let $\MESe{\tau}(\mu) = e \range_E(f) \MESv{\tau}(\alpha)$. 
Define a total preorder on the edges $E_\tau^1$ by $\mu \leq \nu$ if and only if $\MESe{\tau}(\mu) \leq \MESe{\tau}(\nu)$ lexicographically. 
\end{definition}
\noindent


\begin{lem}\label{lem_enumeration}
If $(E_\tau,\LL_\tau)$ is synchronizing in the first label, then the preorders given in definitions \ref{def_order} and \ref{def_edge_order} are strict and total orders (i.e.\ enumerations of the vertices and edges, respectively.)
\end{lem}

\begin{proof}
Let $(E_\tau,\LL_\tau)$ be synchronizing in the first label. Then there exists $m \in \N$ such that each $\alpha \in E_\tau^0$ is uniquely determined by the first $2m$ elements of $\MESv{\tau}(\alpha)$. An analogous argument proves the other statement.
\end{proof}

Notice that, rather than defining one global order on $E^{k-1}$ and $E^k$, each order of minimal emitted sequences orders the permutation 
graph vertices and edges according to a \emph{particular} permutation $\tau$, yielding different orders for different $\tau$. 
\begin{lem}
  \label{lem_ordv}
  Let $\tau,\tau'\in Perm(E^k)$ be endpoint-fixing permutations that synchronize in the first label. If
  $\tau' = g_\pi(\tau)$, then 
  $\MESv{\tau'} = \MESv{\tau}\circ \pi^{-1}$.
\end{lem}
\begin{proof}
  By Lemma \ref{lem_fpi}, there is exactly one labeled edge
  $\pi(\alpha)\llarc{e}{f}\pi(\beta)$ of $(E_{\tau'},\LL_{\tau'})$ for
  each labeled edge $\alpha\llarc{e}{f}\beta$ of $(E_\tau,\LL_\tau)$.
  This lifts to a label preserving bijection between the paths over
  $(E_{\tau},\LL_{\tau})$ and the paths over
  $(E_{\tau'},\LL_{\tau'})$. Since the minimal emitted sequences
  depend only on the path labels, this implies that
  $\MESv{\tau}(\alpha) = \MESv{\tau'}(\pi(\alpha))$ for every
  $\alpha\in E_{\tau}^0$.
\end{proof}

Let $(E_\tau,\LL_\tau)$ be a permutation graph of level $k$ that is synchronizing in the first label.
For $u,v \in E^0$ let $O_{u \to v} = \setof{o^i_{u \to v} }{0 \leq i <
\left | E^{k-1}_{u \to v } \right|}$. Define
\begin{equation}\label{eq_def_O}
O^0 = \bigcup_{u,v\in E^0} O_{u \to v}
\qquad \textnormal{ and } \quad 
O^1 = \{ \mu_0 , \ldots, \mu_{\vert E^{k} \vert -1} \}.
\end{equation}
Since $(E_\tau,\LL_\tau)$ is synchronizing in the first label, the orders of the vertices and edges of $E_\tau$ induced by the minimal emitted sequences are strict and total by Lemma \ref{lem_enumeration}. Hence, it is possible to define bijections $\ordv{\tau} \colon E_\tau^0 \to O^0$ and $\orde{\tau} \colon O^1 \to E_\tau^1$ that number the permutation graph vertices and edges according to the ordering. Note that the edges are numbered globally while the vertices are ordered within each class such that the $i\supth$ element of $E^{k-1}_{u \to v}$ is mapped to $o^i_{u \to v} \in O_{u \to v} \subseteq O^0$.

\begin{definition}[Ordered permutation graph]
  \label{def_Otau}
  Given an endpoint-fixing permutation $\tau\in\Perm(E^k)$ that synchronizes in
  the first label, define its \emph{ordered permutation graph} as the
  labeled graph $(O_\tau,\LL_{O_\tau})$, where
\begin{displaymath}
\begin{array}{l c l}
O_\tau^0 = O^0 
&\quad &
\source_{O_\tau} = \ordv{\tau} \circ \source_\tau \circ \orde{\tau} \\ 
O_\tau^1 = O^1   
&\quad & \range_{O_\tau} = \ordv{\tau} \circ \range_\tau \circ \orde{\tau}
\end{array},
\end{displaymath}
and $\LL_{O_\tau} = \LL_\tau \circ \orde{\tau}$. 
\end{definition}
\noindent In other words, $\ordv{\tau}(\alpha) \llarc{e}{f} \ordv{\tau}(\beta)$ is a labeled edge in $(O_\tau,\LL_{O_\tau})$ precisely when
$\alpha \llarc{e}{f} \beta$ is a labeled edge in $(E_\tau,\LL_\tau)$. In particular, the ordered permutation graph $(O_\tau,\LL_{O_\tau})$ presents the same shift space as $(E_\tau, \LL_\tau)$. In fact, all mutually shift space equivalent permutations at the same level yield the same ordered permutation graph,
as shown in the following proposition:

\begin{prop}\label{prop_vlabels}
Let $\tau, \tau' \in \Perm(E^k)$ be  endpoint-fixing. 
If $(E_\tau,\LL_\tau)$ and $(E_{\tau'},\LL_{\tau'})$ are synchronizing in the first label, then the following is equivalent to the statements given in Proposition \ref{prop_SSE}:
\begin{enumerate}
\setcounter{enumi}{5}
\item $(O_\tau, \LL_{O_\tau}) = (O_{\tau'}, \LL_{O_{\tau'}})$.\label{vlabels_Otau}
\end{enumerate}
\end{prop}

\begin{proof}
Proposition \ref{prop_SSE}\enumref{vlabels_fpi} $\Rightarrow$ \enumref{vlabels_Otau}: 
Assume that $\tau' = g_\pi(\tau)$ for some $\pi\in Perm(E^{k-1})$. 
Then by Lemma \ref{lem_fpi}, there is precisely one edge $\pi(\alpha)\llarc{e}{f}\pi(\beta)$ 
in $(E_{\tau'},\LL_{\tau'})$ for each edge $\alpha\llarc{e}{f}\beta$ of $(E_{\tau},\LL_{\tau})$.
By Definition \ref{def_Otau}, the edges of $(O_\tau,\LL_{O_\tau})$ and $(O_{\tau'},\LL_{O_{\tau'}})$ are then respectively 
$\ordv{\tau}(\alpha) \llarc{e}{f} \ordv{\tau}(\beta)$ and
$\ordv{\tau'}(\pi(\alpha)) \llarc{e}{f} \ordv{\tau'}(\pi(\beta))$ for each of the edges
$\alpha\llarc{e}{f}\beta$ in $(E_\tau,\LL_\tau)$. But $\ordv{\tau'}\circ\pi = \ordv{\tau}$ by Lemma \ref{lem_ordv} and Definition \ref{def_Otau}, 
whereby 
$(O_\tau,\LL_{O_\tau}) = (O_{\tau'},\LL_{O_{\tau'}})$.

\enumref{vlabels_Otau} $\Rightarrow$ Proposition \enumref{prop_SSE}(\ref{vlabels_shift}): 
Whenever $\tau$ is synchronizing in the first label, $\orde{\tau}$ lifts to a label-preserving bijection between the
paths over $(E_\tau,\LL_\tau)$ and the paths over $(O_\tau,\LL_{O_\tau})$, whence the language of the shift space presented
by a permutation graph is the same as the language of the corresponding ordered permutation graph. 
The result then follows trivially.
\end{proof}


%
%

\begin{remark}
\label{rem_computations}
A construction analogous to the one presented in Section \ref{sec_composition} 
allows the ordered permutation graph of $\lambda_\tau \circ  \lambda_\eta$ to be
 computed directly from $(O_\tau,\LL_{O_\tau})$ and $(O_\eta,\LL_{O_\eta})$. 
This gives a straightforward way to find
$\lambda_\tau \circ \lambda_\eta$ up to adjunction by a permutative
unitary when $\lambda_\tau$ and $\lambda_\eta$ are know up to
adjunction by permutative unitaries. In this setting, the construction
is simply a composition of the corresponding textile systems
\cite[p.~18]{nasu_memoir}, and this can be calculated as a matrix 
multiplication in an appropriate semiring. Tools from the theory of textile
systems can hence be applied when studying products of permutative automorphisms, and efficient
algorithms for this purpose will be examined in a forthcoming paper. In this paper, 
tools for automatic detection of automorphism order are presented as well.
\end{remark}

\section{Finding outer permutative automorphisms}
\label{sec_Otau_alg}

Because the number of equivalent automorphisms in a shift space equivalence
class at level $k$ is the same as the number of all endpoint-fixing permutations
at level $(k-1)$, it is infeasible in practice to first
construct all automorphisms before picking a canonical one from each class. 
The great benefit of the order introduced in the previous section is
that it not only facilitates direct construction of a particular equivalence class,
but also direct construction of collections of them. 

The following details the recursive construction of all shift space
equivalence classes for the permutative automorphism at level $k$. The
structure is similar to that of Equation \eqref{eq:compute-Etau},
except that a third recursion layer is added to keep track of the
numbering within each vertex class $O_{u \to v}$. 
In addition, edges are now placed in increasing order of 
minimal emitted sequences, incrementally defining the order at
the same time as building the graph. To achieve this, each
added edge $\mu$ to a partially completed ordered permutation graph $\GG$
is required to satisfy an extra ordering condition compared to
Definition \ref{def_isvalid}:

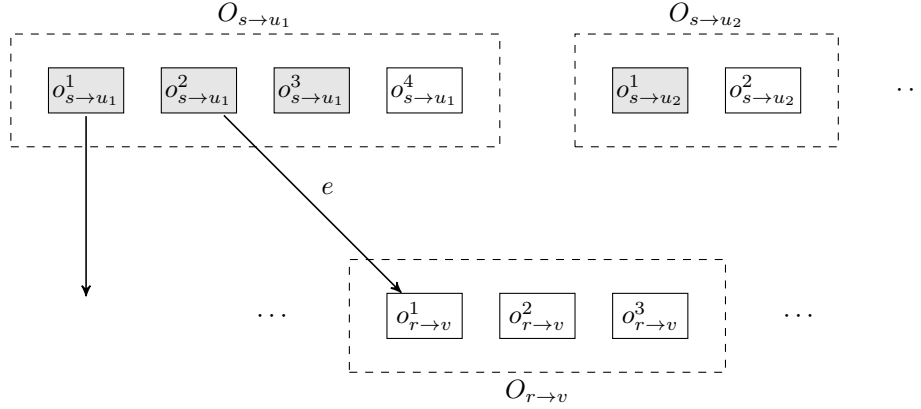
\begin{figure}
  \centering
  \begin{tikzpicture}
  [bend angle=10,
   clearRound/.style = {circle, inner sep = 0pt, minimum size = 17mm},
   clear/.style = {rectangle, minimum width = 10 mm, minimum height = 5 mm, inner sep = 0pt},  
   greyRound/.style = {circle, draw, minimum size = 1 mm, inner sep =
      0pt, fill=black!10},
   grey/.style = {rectangle, draw, minimum width = 10 mm, minimum height=6mm, inner sep =
      1pt, fill=black!10},
    white/.style = {rectangle, draw, minimum width = 10 mm,minimum height=6mm, inner sep =
      1pt},
   to/.style = {->, shorten <= 1 pt, >=stealth', semithick}]

  \node[grey] (s11) at (-6,0) {$o^1_{s \to u_1}$};
  \node[grey] (s12) at (-4.5,0) {$o^2_{s \to u_1}$};
  \node[grey] (s13) at (-3,0) {$o^3_{s \to u_1}$};
  \node[white] (s14) at (-1.5,0) {$o^4_{s \to u_1}$};  
  \node[grey] (s21) at (1.5,0) {$o^1_{s \to u_2}$};
  \node[white] (s22) at (3,0) {$o^2_{s \to u_2}$};

\node (text1) at (-3.75,1) {$O_{s \to u_1}$};
\node (text2) at (2.25,1) {$O_{s \to u_2}$};
\node (text2) at (0,-4) {$O_{r \to v}$};

\node  at (5,0) {$\cdots$};
\node at (3.5,-3) {$\cdots$};
\node at (-3.5,-3) {$\cdots$};

  \node[clear] (r0) at (-6,-3) {};
  \node[white] (r1) at (-1.5,-3) {$o^1_{r \to v}$};
  \node[white] (r2) at (0,-3) {$o^2_{r \to v}$};
  \node[white] (r3) at (1.5,-3) {$o^3_{r \to v}$};

\draw[dashed] (-0.5, 0.75) -- (-7,0.75) -- (-7,-0.75) -- (-0.5,-0.75) -- cycle;
\draw[dashed] (0.5, 0.75) -- (4,0.75) -- (4,-0.75) -- (0.5,-0.75) -- cycle;
\draw[dashed] (-2.5, -2.25) -- (-2.5,-3.75) -- (2.5,-3.75) -- (2.5,-2.25) -- cycle;
 
\draw[to] (s11) to (r0);
\draw[to] (s12) to node[auto] {$e$} (r1);

\end{tikzpicture}
  \caption{ Restriction on the placement of edges caused by Definition \ref{def_isvalid_iii} \enumref{def_isvalid_outdeg}. Previously placed
    edges are as shown. 
    The edge to be placed has first label $e\colon s\to r$ and range
    $o^2_{r\to v}$. The legal sources
    according to Definition
    \ref{def_isvalid_iii}\enumref{def_isvalid_outdeg} are marked with
    gray.
 }
  \label{fig_alg}
\end{figure}

\begin{definition}
  \label{def_isvalid_iii}
  Let $\GG$ be a labeled graph with vertex set $O^0$ and edge set contained in $O^1$.
  Let $e\colon s\to r$ and $f\colon q\to t$ be edges in $E^1$ and
  $\mu\colon o^i_{s\to q} \llarc{e}{f} o^j_{r\to t}$.
  The predicate $\isvalidb{\GG,\mu}$ is true if
  \begin{enumerate}
  \item In $\GG$, $i \le \min\setof{j}{ \mathsf{outdeg}(o^j_{s\to q}) = 0}$.\label{def_isvalid_outdeg} 
  \item $\addedge{\GG}{\mu}$ synchronizes in both labels.\label{def_isvalid_synch}
  \item $\addedge{\GG}{\mu}$ obeys the properties of Proposition \ref{prop_Etau_subgraph}
    when replacing $E^{k-1}_{u \to v}$ everywhere by $O_{u \to v}$.\label{def_isvalid_perm}
  \end{enumerate}  
\end{definition}
\noindent Condition (\ref{def_isvalid_outdeg}) is sketched in Figure \ref{fig_alg}.

In Equation \eqref{eq:compute-Otau}, which is analogous to Equation \eqref{eq:compute-Etau},
$E^0$ is enumerated as $\{v_0,\ldots,v_{|E^0|-1}\}$,
$E^1$ as $\{e_0,\ldots,e_{|E^1|-1}\}$, and $O_{u \to v}$ is defined as in Equation \eqref{eq_def_O}. For each $0 \leq m \leq \vert E^1 \vert-1$, let $r_m$ and $s_m$ denote respectively the range and the source of $e_m$.
Define three mutually recursive functions $\Psi_1$, $\Psi_r$, and $\Psi_s$ by
\begin{equation}
  \label{eq:compute-Otau}
  \begin{split}
  \Psi_1(\vert E^1 \vert,\GG) &= \{\GG\} \\
  \Psi_1(m,\GG) &= \Psi_r(m,0,\GG) \\
  \\
  \Psi_r(m,\vert E^0\vert,\GG) &= \Psi_1(m+1,\GG) \\
  \Psi_r(m,n,\GG) &= \Psi_s\left(m,n,0,\GG\right)\\
  \\
  \Psi_s(m,n,|E^{k-1}_{r_m\to v_n}|,\GG)  &= \Psi_r(m,n+1,\GG)\\
  \Psi_s(m,n,i,\GG)  &= \!\!\! \!\!\! \disjointunion_{
      \begin{smallmatrix}
      {\mu_j\colon \beta\,\llarc{e_m}{f}\,o_{r_m\to v_n}^i}\\
      {\isvalidb{\GG,\mu_j}}
    \end{smallmatrix}
  }
  \!\!\! \!\!\! \Psi_s(m,n,i+1,\GG\oplus\mu_j), \quad j = \left \vert G^1 \right \vert
  \end{split}
\end{equation}
The algorithm proceeds as follows:
For each $e_m\colon s_m\to r_m$ in order,
together with each range $o^i_{r_m\to v_n} \in O_{r_m\to v_n}$ in
order, all possible ways of placing edges of the form $\mu_j\colon
\beta\llarc{e_m}{f} o^i_{r_m\to v_n}$ that satisfy Definition
\ref{def_isvalid_iii} are tried, taking the disjoint union of 
the results from the recursion level below. It will be shown in the following that 
when $\GG_0$ is the empty labeled graph with vertex set $O^0$,
$\completeOtau{0,\GG_0}$ is exactly the set of ordered permutation
graphs of every shift space equivalence class for permutative 
automorphisms at level $k$, and that each equivalence class is 
constructed exactly once.

\begin{lem}\label{lem_construction_ranges}
Let $(G,\LL) \in\completeOtau{0,\GG_0}$ be a labeled graph constructed by the algorithm presented above
Let $\mu_i, \mu_{i'} \in G^1$ with labels $[e,f]$ and $[e',f']$, respectively. Then $i < i'$ if and only if one of the following three conditions holds:
\begin{enumerate}
\item $e < e'$.
\item $e = e'$ and $\range_E(f) < \range_E(f')$.
\item $e = e'$, $\range_E(f) = \range_E(f')$ and $j < j'$ when $\range_G(\mu_i) = o^j_{\range_E(e) \to \range_E(f)}$ and $\range_G(\mu_{i'}) = o^{j'}_{\range_E(e) \to \range_E(f)}$. 
\end{enumerate}
\end{lem}

\begin{proof}
It is easy to see from the structure of Equation \eqref{eq:compute-Etau} that the triplets $(m,n,i)$ are visited in ascending lexicographic order, whereby edges are placed ordered first by first label $e_m$, then by $\range_E(f) = v_n$, and lastly by the position $i$ within $O_{r_m\to v_n}$. In case $e\range_E(f) = e'\range_E(f')$, it follows that $\mu_i$ was placed before $\mu_{i'}$ if and only if the range of $\mu_i$ has a lower number than the range of $\mu_{i'}$ in $O_{\range_E(e) \to \range_E(f)}$.
\end{proof}

For a labeled graph $\GG = (G,\LL)$ constructed by the algorithm presented above, i.e.\ for a $\GG\in\completeOtau{0,\GG_0}$, the $i\supth$ edge $\mu_i \in G^1$ is said to be a \emph{first edge} if there is no $j < i$ such that $\source_G(\mu_j) = \source_G(\mu_i)$. That is, $\mu_i$ was the first edge with source $\source_G(\mu_i)$ added in the algorithmic construction. Clearly, there is precisely one first edge associated to each vertex.

\begin{lem}\label{lem_construction_sources}
Let $(G,\LL) \in\completeOtau{0,\GG_0}$, and
let $\mu_i, \mu_{i'} \in G^1$ be first edges with labels $[e,f]$ and $[e',f']$, respectively. Assume that $\source_E(e) = \source_E(e')$ and  $\source_E(f) = \source_E(f')$. Let $\source_G(\mu_i) = o^j_{\source(e) \to \source(f)}$ and $\source_G(\mu_{i'}) = o^{j'}_{\source(e) \to \source(f)}$.
Then $i < i'$ if and only if $j < j'$.
\end{lem}

\begin{proof}
This follows immediately from Definition \ref{def_isvalid_iii}\enumref{def_isvalid_outdeg}. 
\end{proof}

\begin{lem}\label{lem_first_added}
Let $(G,\LL) \in\completeOtau{0,\GG_0}$, let $o \in O^0$, and
let $M = \mu_{i_1} \mu_{i_2} \cdots $ be the unique path with source $o$ consisting only of first edges.
Then $M$ is the unique minimal emitted path of $o$.
\end{lem}


\begin{proof}
Let $o \in O^0$. By Definition \ref{def_isvalid_iii}\enumref{def_isvalid_synch}, $\GG$ is synchronizing in the first label, so the minimal emitted path of $o$ is unique by Lemma \ref{lem_enumeration}. Let $M' = \mu_{i_1'} \mu_{i_2'} \cdots $ be this minimal emitted path  

Consider $\mu = \mu_{i_1}$ and $\mu' = \mu_{i_1'}$. Let $\LL(\mu) = [e,f]$ and let $\LL(\mu') = [e',f']$.
Since $M'$ is the minimal emitted path of $o$, $e' \leq e$. Since $\mu$ was the first edge added with source $o$, Lemma \ref{lem_construction_ranges} implies $e \leq e'$ and hence $e = e'$. Similarly,  $\range_E(f') \leq \range_E(f)$ because $M'$ is a minimal emitted path, and Lemma \ref{lem_construction_ranges} implies $\range_E(f) \leq \range_E(f')$, whereby $\range_E(f') = \range_E(f)$.

Furthermore, Lemma \ref{lem_construction_ranges} implies that the number of $\range_G(\mu)$ in $O_{\range(e) \to \range(f)}$ is lower than or equal to the number of $\range_G(\mu')$. Hence, Lemma \ref{lem_construction_sources} implies that the first edge with source $\range_G(\mu)$ was added before the first edge with source $\range_G(\mu')$, i.e.\ $i_2 \leq i_2'$. Repeated applications of this argument proves that for any $n$, $\LL_1(\mu_{i_1} \mu_{i_2} \cdots \mu_{i_n}) = \LL_1(\mu_{i_1'} \mu_{i_2'} \cdots \mu_{i_n'})$. Since $\GG$ is synchronizing in the first label, this implies that $M = M'$.
\end{proof}

%

\begin{lem}\label{lem_construction_order}
Let $(G,\LL) \in\completeOtau{0,\GG_0}$,
let $o,o' \in O^0$, 
and let their minimal emitted paths be $\mu_{i_1} \mu_{i_2} \cdots$ and $\mu_{i_1'} \mu_{i_2'} \cdots$, respectively. Then $o < o'$ in the order of minimal emitted sequences if and only if $i_1 < i_1'$.
\end{lem}

\begin{proof}
Assume $o \neq o'$. For each $i$, let $[e_i,f_i]$ and $[e_j',f_j']$ be the labels of $\mu_{i_j}$ and $\mu_{i_j'}$, respectively. Note that the edges $\mu_{i_j}$ and $\mu_{i'_j}$ are all first edges by Lemma \ref{lem_first_added}.
By Definition \ref{def_isvalid_iii}\enumref{def_isvalid_synch}, $\GG$ is synchronizing in the first label, so by Lemma \ref{lem_enumeration}, the two minimal emitted sequences are different. Choose the minimal $l \in \N$ for which $e_l \range_E(f_l) \neq e_l' \range_E(f_l')$. 
If $l = 1$%
, then Lemma \ref{lem_construction_ranges} trivially implies the result, so assume that $l >  1$.

Assume that $i_1 < i_1'$. Since $e_1\range_E(f_1) = e_1'\range_E(f_1')$, Lemma \ref{lem_construction_ranges} implies that $\range_G(\mu)$ has a strictly lower number in $O_{\range_E(e) \to \range_E(f)}$ than $\range_G(\mu')$. Hence, Lemma \ref{lem_construction_sources} implies that $i_2 < i_2'$. Repeated applications of this argument leads to $i_l < i_l'$. Since $e_l \range_E(f_l) \neq e_l' \range_E(f_l')$, it follows that Lemma \ref{lem_construction_ranges} yields $e_l \range_E(f_l) < e_l' \range_E(f_l')$.

For the converse implication, assume that the minimal emitted sequence of $o$ is strictly smaller than the minimal emitted sequence of $o'$. Then $e_l \range_E(f_l) < e_l' \range_E(f_l')$, so $i_l < i_l'$ by Lemma \ref{lem_construction_ranges}.  Since $e_j \range_E(f_j) = e_j' \range_E(f_j')$ for all $j < l$, Lemmas \ref{lem_construction_ranges} and \ref{lem_construction_sources} yield $i_1 < i_1'$. 
%
\end{proof}

\begin{cor}\label{cor_order_of_output}
Let $(G,\LL) \in\completeOtau{0,\GG_0}$.
The following hold:
\begin{enumerate}
\item For $o=o^j_{u \to v},o'= o^{j'}_{u \to v} \in O_{u \to v}$, $o < o'$ in the order of minimal emitted sequences if and only if $j < j'$.
\item For $\mu_i, \mu_i' \in G^1$, $\mu_i < \mu_{i'}$ in the order of minimal emitted sequences if and only if $i < i'$. 
\end{enumerate}
\end{cor}

\begin{proof}\
\begin{enumerate}
\item Let $\mu_i$ and $\mu_{i'}$ be the first edge emitted by respectively $o$ and $o'$. By Lemma \ref{lem_first_added} and Lemma \ref{lem_construction_order}, $o < o'$ if and only if $i < i'$. Since $\mu_i$ and $\mu_{i'}$ are first edges, the result then follows from Lemma \ref{lem_construction_sources}.
\item This follows from Lemma \ref{lem_construction_ranges} and the previous statement.
\end{enumerate}
\end{proof}

\begin{thm}
  \label{thm_outer_construction}
  Let $\GG_0$ be the labeled graph with vertex set $O^0$ and no edges.
  $\completeOtau{0 ,\GG_0}$ is the set of ordered permutation graphs
  corresponding to the shift space equivalence classes for every permutative automorphism at level $k$. Furthermore, each ordered permutation graph is constructed only once. 
\end{thm}

\begin{proof}
Let $\GG\in\completeOtau{0,\GG_0}$ be a labeled graph constructed by the algorithm. The aim is to show that $\GG$ is an ordered permutation graph. By an argument analogous to the one used in the proof of Theorem \ref{thm_Pgraph_construct}, Definition \ref{def_isvalid_iii}\enumref{def_isvalid_perm} implies that any assignment of the elements of $E^{k-1}_{u \to v}$ to $O_{u \to v}$ will map $\GG$ to a permutation graph. Furthermore, by Definition \ref{def_isvalid_iii}\enumref{def_isvalid_synch}, the resulting permutation graph is synchronizing in both labels and hence corresponds to an automorphism. By Corollary \ref{cor_order_of_output}, the vertices and edges of $\GG$ are ordered correctly, and consequently $\GG$ is an ordered permutation graph for an automorphism. 

Next, let $\GG = (G, \LL)$ be an ordered permutation graph, corresponding to the shift space equivalence class for a permutative automorphism. The aim is to show that $\GG\in\completeOtau{0,\GG_0}$.
Given $0 \leq j \leq |G^1|$, let $\GG_j$ denote the subgraph of $\GG$ obtained by including precisely the first $j$ edges of $G^1$ in the order of minimal emitted sequences. Note that for $j >0$, the $j$ edges in $\GG_j$ are $\mu_0, \ldots, \mu_{j-1}$.
The following proves by induction that for each $j$, the subgraph $\GG_j$ is constructed by an intermediate step of the algorithm. 
Note that the algorithm starts from $\GG_0$. 

For the induction step, assume that $\GG_{j}$ has been constructed at an intermediate step.
Let $\mu_j \in G^1$ with $\mu_j\colon o^l_{s \to u} \llarc{e_m}{f} o^i_{r \to v_n}$. 
Consider another edge $\mu_{j'} \in G^1$ for which $\mu_{j'}\colon o^{l'}_{s' \to u'} \llarc{e_{m'}}{f'} o^{i'}_{r' \to v_{n'}}$.
Since $\GG$ is an ordered permutation graph, $j' < j$ if and only if $(n',m',i') < (n,m,i)$ lexicographically. Specifically, the edges in $\GG_{j}$ are precisely the $\mu_{j'} \in G^1$ for which $(n',m',i') < (n,m,i)$. Hence, the next step in the algorithm, following the construction of $\GG_{j}$ will add an edge with the same first label and range as $\mu_j$. The specific edge $\mu_j$ can be added by the algorithm if and only if $(\GG_{j},\mu_j)$ satisfies Definition \ref{def_isvalid_iii}. Since $\GG$ is an ordered permutation graph, $(\GG_{j},\mu_j)$ satisfies Definition \ref{def_isvalid_iii} \enumref{def_isvalid_synch} and \enumref{def_isvalid_perm}.

It remains to be shown that $(\GG_{j},\mu_j)$ satisfies Definition \ref{def_isvalid_iii} \enumref{def_isvalid_outdeg}, i.e.\ that $o^l_{s \to u}$ is a valid source in the construction. Assume there exists a $o^{l'}_{s \to u}$ with $l' < l$ (the condition is trivially satisfied if no such $l'$ exists). Since $\GG$ is an ordered permutation graph, $o^{l'}_{s \to u} < o^{l}_{s \to u}$ in the order of minimal emitted sequences. Hence, there is an edge $\mu_{j'} \in G^1$ with source $o^{l'}_{s \to u}$ and $j' < j$, so  $o^{l'}_{s \to u}$ emits at least one edge in $\GG_{j}$. Consequently, $(\GG_{j},\mu_j)$ satisfies Definition \ref{def_isvalid_iii} \enumref{def_isvalid_outdeg},
whereby $\GG_{j} \oplus \mu_j = \GG_{j+1}$ can be constructed by an intermediate step of the algorithm. Because $\GG = \GG_{|G^1|}$, the procedure terminates, and the desired result follows.

Let $\GG$ be an ordered permutation graph. To see that the algorithm does not construct $\GG$ twice, note that it places edges in ascending order. Given a subgraph $\GG_j$ as above, there is precisely one way to place the next edge to obtain $\GG_{j+1}$.  
\end{proof}

\section{Example: Constructing automorphisms} 
\label{sec_bowtie}
Consider the graph $E$ shown in Figure \ref{fig_bowtie}. In this
section, the techniques developed in the previous sections will be
used to investigate the permutative endomorphisms and automorphisms of
$C^\ast(E)$. For each $k$, there are clearly $3\cdot2^k$ paths of
length $k$ in $E$. For small values of $k$, Table \ref{table_bowtie}
summarizes the number of paths, permutative endomorphisms, permutative
automorphisms, and classes of permutative automorphisms equivalent up
to adjunction by permutative unitaries as described in Proposition
\ref{prop_vlabels}.
The numbers in the first three columns are easily computed, while
the last two columns require finding all the automorphisms. Notice the rapid combinatorial growth with $k$. Note also that the
quotient of the numbers in columns 4 and 5 is the number of
permutative endomorphisms of the previous level, as can be understood from Proposition \ref{prop_SSE}.

\begin{figure}
\begin{center}
\begin{tikzpicture}
  [bend angle=10,
  clearRound/.style = {circle, inner sep = 0pt, minimum size = 17mm},
  clear/.style = {rectangle, minimum width = 5 mm, minimum height = 5 mm, inner sep = 0pt},  
  greyRound/.style = {circle, draw, minimum size = 1 mm, inner sep =
    0pt, fill=black!10},
  grey/.style = {rectangle, draw, minimum size = 6 mm, inner sep =
    1pt, fill=black!10},
  white/.style = {rectangle, draw, minimum size = 6 mm, inner sep =
    1pt},
  to/.style = {->, shorten <= 1 pt, >=stealth', semithick}]  
  
  \node[grey] (v1) at (0,0) {$v_1$};
  \node[grey] (v2) at (2,0) {$v_2$}; 
  \node[grey] (v3) at (4,0) {$v_3$};  

  \draw[to,loop left] (v1) to node[auto] {$a$} (v1);
  \draw[to, bend left] (v1) to node[auto] {$b$} (v2);
  \draw[to, bend left] (v2) to node[auto] {$c$} (v1);
  \draw[to, bend left] (v2) to node[auto] {$d$} (v3);
  \draw[to, bend left] (v3) to node[auto] {$e$} (v2);
  \draw[to, loop right] (v3) to node[auto] {$f$} (v3);
  
\end{tikzpicture}
\end{center}
\caption{The graph $E$ considered in Section \ref{sec_bowtie}.} 
\label{fig_bowtie}
\end{figure}
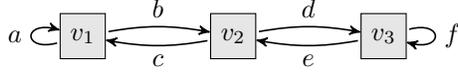

\begin{table}
\begin{center}
\begin{tabular}{c r r c l r r}
\toprule
$k$ & Paths & \multicolumn{3}{ c }{\quad \quad \; Endomorphisms} & Automorphisms & Classes \\ 
\toprule
1 & 6 & & 1 & & 1 & 1 \\
2 & 12 & $(2!\cdot1! \cdot 1!)^3$ &$=$& 8 & 2 & 2 \\
3 & 24 & $(2!\cdot3! \cdot 3!)^3$ &$=$&  373,248  & 32 & 4 \\
4 & 48 & $(5!\cdot5! \cdot 6!)^3$ &$\approx$& $1.1\cdot 10^{21}$ & 454,989,312  & 1,219 \\
5 & 96 & $(10!\cdot 11! \cdot 11!)^3$ &$\approx$& $1.9 \cdot 10^{65}$ & ? & ? \\
\toprule
\end{tabular}
\caption{ 
  Number of paths of length $k$, permutative endomorphisms,
  permutative automorphisms, and order classes of permutative automorphisms
  for small values of $k$.
}
\label{table_bowtie}
\end{center}
\end{table}

At levels $k=1$ and $k=2$, it is possible to construct the
ordered permutation graphs by hand using the algorithms and tools
given above.
At level $k=3$, there are too many permutative endomorphisms to check
by hand, however, it would be possible to test
all the permutative endomorphisms individually using a direct brute
force computer program, but it would be time consuming. Using the algorithm from Section \ref{sec_Otau_alg},
the search is completed instantaneously.
At level $k=4$, there are so many endpoint-fixing
permutations that the problem cannot be handled by brute force.  Even
assuming that the automorphism condition can be tested in a
microsecond, considering all permutative endomorphisms individually
would take 35 million years.  However, combining the methods presented
in this paper, it is possible to reduce the problem to a tractable
one: Constructing representations of the equivalence classes directly
using the methods of Section \ref{sec_inner_order} reduces the number of
graphs to consider by a factor 373,248 from $1.1\cdot 10^{21}$ to
$3.0\cdot 10^{15}$, so -- under the same assumption as above --
investigating each class would take about 100 years. However, the pruning of the search tree built into the algorithm presented in Section \ref{sec_Otau_alg} means that 
we are able to complete the exhaustive search for 
automorphisms classes in minutes, and this investigation reveals that there are 1219 distinct equivalence classes. 

There is, however, no way to avoid the combinatorial growth of the
problem: At level $k=5$, the number of permutative
endomorphism classes is $6\cdot 10^{28}$ times greater than at level
$k=4$, and hence, we expect the number of automorphism classes to be so
large that it is impossible to generate all of them. For level $k=5$
and beyond, we are thus forced to restrict our attention to
interesting subsets. The following sections give detailed discussions
of the automorphisms found in the levels $k=1$ through $k=5$.

\subsection*{Levels $k=1$ and $k=2$}
Clearly, the identity is the only endpoint-fixing permutation at level
$k=1$, and it is straightforward to check by hand that it induces an
automorphism.  At level $k=2$, there are precisely two permutative
automorphisms. They are given by the following two endpoint-fixing
permutations of $E^2$:
\begin{displaymath} \tau_{2,0} = \Id \qquad,
  \qquad \tau_{2,1} = (de, cb).
\end{displaymath}

\subsection*{Level $k=3$}
At this level, there are four different equivalence classes of
permutative automorphisms. Representatives of these classes are given
by the following endpoint-fixing permutations of $E^3$:
\begin{align*}
\tau_{3,0} &=  \Id \\
\tau_{3,1} &= (d\f e, cab) \\
\tau_{3,2} &= (bde, bcb) (ede, ecb) \\
\tau_{3,3} &= (bde, bcb) (d\f e, cab) (ede, ecb)
\end{align*} 
Each class contains 8 permutative automorphisms, since every
endpoint-fixing permutation from the previous level gives a way to
permute the labels of the permutation graph resulting in a
presentation of a permutative automorphism that is inner equivalent to
the original one as described in Proposition \ref{prop_vlabels}.  Note that $\tau_{3,1}$ induces an automorphism inner
equivalent to $\tau_{2,1}$. In general, each permutative automorphism
at level $k$ will give rise to inner equivalent permutative automorphism at level
$k+1$ in this way.

\subsection*{Level $k=4$}
At this level, there are 1219 classes of permutative automorphisms,
which is too many to list. As an example, consider the permutation
\begin{displaymath}
\tau_4 = (abde,bd\f e,bcab)(ecab,ed\f e).
\end{displaymath}
This gives a permutative automorphism that occurs at this level
without being equivalent to one occurring at a of lower level. This
permutative automorphism is of interest because 
it has infinite order, unlike the permutative automorphisms at levels
$1$ through $3$.

\subsection*{Level $k=5$}
As mentioned above, it is not feasible to find all permutative
automorphisms at level $k=5$ using the methods developed in this
paper. However, by restricting to automorphisms with certain
properties, it is possible to find interesting automorphisms by brute
force. Define 
\begin{displaymath}
T_1 = S_a + S_b + S_d \quad \textnormal{ and } \quad
T_2 = S_c + S_e +S_f.
\end{displaymath} 
Then it is straightforward to check that $T_1,T_2$ generate a unital
copy of ${\mathcal O}_2$ inside $C^*(E)$. 
The aim of the following
will be to identify permutative automorphisms of $C^*(E)$ that fix
$T_2$, and that send $T_1$ to a sum of words in $T_1, T_2, T_1^*, T_2^*$. 
Such a permutative automorphism of $C^*(E)$ will
naturally induce an automorphism of ${\mathcal O}_2$. However, even
with this restriction, there are too many possibilities to check by
brute force. To counter this, we restrict to permutations that fix any
path (of length 5) starting with $c$, $e$ or $f$ (i.e.\ the induced
automorphism of ${\mathcal O}_2$ fixes all paths starting with the
second edge). This restricts the problem to an investigation of
$5!\cdot10! = 435456000$ permutations, and this number is sufficiently
small to be handled by brute force.  This investigation gives 12
distinct $\OO_2$-preserving permutative automorphisms that fix all paths
starting with $c$, $e$ or $f$. As an example, two of these are:
\begin{displaymath}
\tau_5 = (abcbc,abdec,bd\f ec)(abd\f e,bd\f \f e,bdede)(abd\f \f ,bd\f \f \f ,bded\f )
\end{displaymath}
and
\begin{displaymath}
\tau_5' = (aaabc,abdec,abcbc,bd\f ec)(aabde,abd\f e,bdede,bd\f \f e)(aabd\f ,abd\f \f ,bded\f ,bd\f \f \f ).
\end{displaymath}

Evidence suggests that $\phi_{\tau_5}$ has infinite order while $\phi_{\tau_5'}$ can be proved to be of order 60. General tools for examining the order of permutative automorphisms will be examined in the forthcoming paper mentioned in Remark \ref{rem_computations}.


\section{Example: Automorphisms of $\OO_n$}
\label{sec:On}

\newcommand{\ift}{$\infty$}
\begin{table}
\begin{center}
\begin{tabular}{l cccccc}
\toprule
Permutation $\tau$ & \multicolumn{6}{c}{Orders of induced automorphisms} \\ 
\toprule
      &  $\phi_\tau$ & $\phi_{(a,b)}\phi_\tau$ & $\phi_{(a,c)}\phi_\tau$ & $\phi_{(b,c)}\phi_\tau$ &$\phi_{(a,b,c)}\phi_\tau$ & $\phi_{(a,c,b)}\phi_\tau$\!\!\!\\
      \cline{2-7}\\[-0.5em]
$\Id$                             & 1   & 2  & 2  & 2  & 3  & 3  \\
$(ca,cb)$                         & 2   & 2  &\ift&\ift&\ift&\ift\\
$(bb,bc)(cb,cc)$                  & 2   &\ift&\ift& 2  &\ift&\ift\\
$(bb,cb,ca)(bc,cc)$               &\ift &\ift&\ift&\ift&\ift&\ift\\
$(bb,cb)(bc,cc)$                  & 2   &\ift&\ift& 2  &\ift&\ift\\
$(bb,cc)(bc,cb)$                  & 2   &\ift&\ift& 2  &\ift&\ift\\
$(ba,bc)$                         & 2   &\ift& 2  &\ift&\ift&\ift\\
$(ba,cc,bc)(bb,cb)$               &\ift &\ift&\ift&\ift&\ift&\ift\\
$(ac,bc,ba)$                      &\ift &\ift&\ift&\ift&\ift&\ift\\
$(ac,cc,bc,ca,ba)(bb,cb)$         &\ift &\ift&\ift&\ift&\ift&\ift\\
$(ac,bc,bb)(cb,cc)$               &\ift &\ift&\ift&\ift&\ift&\ift\\
$(ac,cc,bb)(ba,ca)(bc,cb)$        &\ift &\ift&\ift&\ift&\ift&\ift\\
$(ac,bc)$                         & 2   &2   &\ift&\ift&\ift&\ift\\
$(ac,bc)(ca,cb)$                  & 2   &2   &\ift&\ift&\ift&\ift\\
$(ac,cc,bc)(ba,ca)(bb,cb)$        &\ift &\ift&\ift& 2  &\ift& 2  \\
$(ac,cc,bc)(ba,ca,bb,cb)$         &\ift &\ift&\ift& 2  &\ift& 2  \\
\toprule
\end{tabular}
\caption{ 
Permutations that induce distinct outer permutative automorphisms at level 2 in $\OO_3$, and the orders of the 96 shift space automorphisms. The first column lists the 16 permutations that induce distinct automorphisms of $\OO_3$ up to inner equivalence and graph automorphism. For each such $\tau$, its row lists the orders of the 6 shift space automorphisms obtained by composing $\phi_\tau$ with the shift space automorphisms induced by graph automorphisms, i.e., the 6 permutations of the three edges in the graph for $\OO_3$. The orders are computed using the techniques mentioned in Remark \ref{rem_computations}. 
}
\label{table_O3_2}
\end{center}
\end{table}

The permutative automorphisms of $\OO_n$ have already been investigated experimentally in \cite{conti_kimberley_szymanski,conti_szymanski_labeled_trees}, and it is straightforward to use the algorithms presented in Sections \ref{sec_Etau_alg} and \ref{sec_Otau_alg} to verify the numbers of pertutative automorphisms of $\OO_n$ and classes of such automorphisms found in these two papers. It is worth noting that the algorithms presented in the present paper yield significant improvements over the specialized approaches used in the previous investigations. First of all, the methods presented here can be applied to a wide range of interesting graph algebras beyond the Cuntz algebras. Secondly, they are much faster than the previously used methods. For instance, the original identification of the automorphism classes at level 2 in $\OO_4$ took approximately 70 days of computation on a server \cite{conti_szymanski_labeled_trees} while our methods were able to deliver the same result after two seconds of computation on a standard laptop, an improvement by a factor of more than two million. 

As an example of the results achieved in this way, consider the permutative automorphisms at level 2 in $\OO_3$.  In this case, there are 96 outer permutative automorphisms, but the entire collection can be reconstructed from the 16 permutations listed in Table \ref{table_O3_2} by composing with the graph automorphisms of the original graph, i.e.\ the $3! = 6$ permutations of the three edges. As in the previous example, the techniques mentioned in Remark \ref{rem_computations} were used to investigate the orders of these permutative automorphisms, and the results of this investigation are listed in Table \ref{table_O3_2}. The exhaustive search for outer permutative automorphisms at level 2 in $\OO_3$ required $2\mathrm{ms}$ of computation time, and the automorphism orders took $60\mathrm{ms}$ in total to calculate. 


%

In spite of the drastric improvement in the speed of the computations achieved through the techniques developed in this paper, it may still not be feasible to extend the investigation of the Cuntz algebras from \cite{conti_kimberley_szymanski,conti_szymanski_labeled_trees} to larger values of $n$, or to higher levels, due to the violent combinatorial growth of the problem.



\end{document}